\documentclass{amsart}
\usepackage{amsfonts,amssymb,amsmath,amsthm}
\usepackage{url}
\usepackage{enumerate}
\usepackage{hyperref}
\usepackage{fancyhdr}
\newtheorem{theorem}{Theorem}[section]
\newtheorem{lemma}[theorem]{Lemma}

\newtheorem{corollary}[theorem]{Corollary}
\newtheorem{example}[theorem]{Example}

\theoremstyle{definition}
\newtheorem{definition}[theorem]{Definition}
\newtheorem{definitions}[theorem]{Definitions}

\newtheorem{remark}[theorem]{Remark}
\newtheorem{remarks}[theorem]{Remarks}
\def\N{\mathbb{N}}
\def\R{\mathbb{R}}
\def\C{\mathbb{C}}
\def\Q{\mathbb{Q}}
\def\Z{\mathbb{Z}}
\def\F{\mathbb{F}}
\def\lien{\mathrel{\mkern-4mu}}
\def\too{\relbar\lien\rightarrow}
\def\tooo{\relbar\lien\relbar\lien\too}

\let\ds=\displaystyle
\let\wt=\widetilde
\def\Cl{{\mathcal C}\hskip-2pt{\ell}}
\def\cl{c\hskip-1pt{\ell}}
\def\Pl{P\hskip-1pt{l}}
\def\plus{\ds\mathop{\raise 2.0pt \hbox{$\bigoplus$}}\limits}
\def\prd{ \ds\mathop{\raise 2.0pt \hbox{$\prod$}}\limits}
\def\sm{  \ds\mathop{\raise 2.0pt \hbox{$\sum$}}\limits}

\author[Georges  Gras]{Georges  Gras}
\address{Villa la Gardette \\ chemin Ch\^ateau Gagni\`ere \\  F--38520 Le Bourg d'Oisans.}
\email{g.mn.gras@wanadoo.fr  {\it url\,:\,}\url{http://www.researchgate.net/profile/Georges_Gras} }



\begin{document}
 
\title[Invariant generalized ideal classes]{Invariant generalized ideal classes \\ Structure theorems for
$p$-class groups \\ in $p$-extensions \\ \vspace{0.2cm} {\footnotesize \it A Survey}}

\date{October 14, 2016}

\begin{abstract}
We give, in Sections 2 and 3, an english translation of: {\it Classes g\'en\'e\-ralis\'ees 
invariantes}, J. Math. Soc. Japan, 46,  3 (1994), with some improvements and with notations
and definitions in accordance with our book:  {\it Class Field Theory: from theory to practice},
SMM, Springer-Verlag, $2^{\rm nd}$ corrected printing 2005. 
We recall, in Section 4, some  structure theorems for finite $\Z_p[G]$-modules ($G \simeq \Z/p\,\Z$) obtained in: 
 {\it Sur les $\ell$-classes d'id\'eaux dans les extensions \!cycliques relatives de degr\'e premier \!$\ell$},
Annales de l'Institut Fourier,  23,  3 (1973). Then we recall the algorithm of local normic computations which allows 
to obtain the order and (potentially) the structure of a $p$-class group in a cyclic extension of degree $p$.

\noindent
In Section 5, we apply this to the study of the structure of relative $p$-class groups of Abelian extensions 
of prime to $p$ degree, using the Thaine--Ribet--Mazur--Wiles--Kolyvagin ``principal theorem'', and the 
notion of ``admissible sets of prime numbers'' in a cyclic extension of degree $p$, from:
{\it Sur la structure des groupes de classes relatives}, Annales de l'Institut Fourier, 43, 1 (1993).

\noindent
In conclusion, we suggest the study, in the same spirit, of some deep invariants attached to the 
$p$-ramification theory (as dual form of non-ramification theory) and which have become standard 
 in a $p$-adic framework.

\noindent
Since some of these techniques have often been rediscovered, we give a substantial (but certainly incomplete) 
bibliography which may be used to have a broad view on the subject.
\end{abstract}

\maketitle

\section{Introduction -- Generalities} \label{sect1}

Let $K/k$ be a cyclic extension of algebraic number fields, with Galois group $G$, and let $L$ be a
finite Abelian extension of $K$; we suppose that $L/k$ is Galois, so that $G$ operates by conjugation
on ${\rm Gal}(L/K)$. 

\smallskip
We shall see the field $L$ given, via Class Field Theory, by some Artin group of $K$ 
(e.g., the Hilbert class field $H_K^{+}$ of $K$ associated with 
the group of principal ideals, in the narrow sense,
any ray class field $H_{K, {\mathfrak m}}^{+}$ associated with a ray group 
modulo a modulus ${\mathfrak m}$ of $k$, in the narrow sense, or more generally any 
subfield $L$ of these canonical fields, defining ${\rm Gal}(H_{K, {\mathfrak m}}^{+}/L)$ by means of a
sub-$G$-module ${\mathcal H}$ of the generalized class group 
$\Cl_{K, {\mathfrak m}}^{+} \simeq {\rm Gal}(H_{K, {\mathfrak m}}^{+}/K)$). 

\smallskip
We intend to give, from the arithmetic of $k$ and elementary local normic 
computations in $K/k$, an explicit formula for
$$\#  {\rm Gal}(L/K)^G = \#  (\Cl_{K, {\mathfrak m}}^{+}/ {\mathcal H} )^G. $$

This order is the degree, over $K$, of the maximal subfield of $L$ (denoted $L^{\rm ab}$) which is 
Abelian over $k$.

\smallskip
Indeed, since $G$ is cyclic, it is not difficult to see that the com\-mutator subgroup
$[\Gamma, \Gamma]$ of $\Gamma :=  {\rm Gal}(L/k)$ is equal to ${\rm Gal}(L/K)^{1-\sigma} \simeq
(\Cl_{K, {\mathfrak m}}^{+}/ {\mathcal H} )^{1-\sigma}$,
where $\sigma$ is a generator of $G$ (or an extension in $\Gamma$).
So we have the exact sequences
\begin{equation}\label{eq1}
\begin{aligned}
&1 \too {\rm Gal}(L/K)^{1-\sigma} \tooo \Gamma \tooo \Gamma^{\rm ab} = 
\Gamma/[\Gamma, \Gamma] = {\rm Gal}(L^{\rm ab}/k) \too 1,   \\
&1 \too {\rm Gal}(L/K)^G \tooo {\rm Gal}(L/K) \mathop{\tooo}^{1-\sigma} {\rm Gal}(L/K)^{1-\sigma} \too 1.
\end{aligned}
\end{equation}

 Hence $\ds \#  {\rm Gal}(L/K)^G = [L : K] \cdot \frac{\# \Gamma^{\rm ab}}{\# \Gamma}
= [L : K] \cdot \frac{ [L^{\rm ab} : k]}{[L : K] [K : k]} = [L^{\rm ab} : K]$.
The study of the structure of ${\rm Gal}(L/K)$ as $G$-module
(or at least the computation of its order) is based under the study of the following filtration:

\begin{definition} \label{def1}
Let $M := {\rm Gal}(L/K)$ and let $(M_i)_{i \geq 0}$ be the 
increasing sequence of sub-$G$-modules defined (with $M_0:= 1$) by
$$M_{i+1}/M_i := (M/M_i)^G,\ \, \hbox{for $0\leq i \leq n$}, $$
where $n$ is the least integer $i$ such that $M_i = M$. 
\end{definition}

For $i=0$, we get $M_1 = M^G$. 
We have equivalently $M_{i+1}= \{h \in M, \ h^{1-\sigma} \in M_i\}$.
Thus $M_i= \{h \in M, \ h^{(1-\sigma)^i} =1 \}$ and $(1-\sigma)^n$ is the anihilator of $M$.

\smallskip
If $L_i$ is the subfield of $L$ fixed by $M_i$, this yields the following tower of fields, Galois over $k$, 
from the exact sequences $1 \to M_i \tooo M\ds \mathop{\tooo}^{(1-\sigma)^i} M^{(1-\sigma)^i} \to 1$
such that $[L_i : L_{i+1}] = (M_{i+1} : M_i)$ which can be computed from local arithmetical tools in $K/k$
as described in the Sections \ref{sect3} and \ref{sect4}:

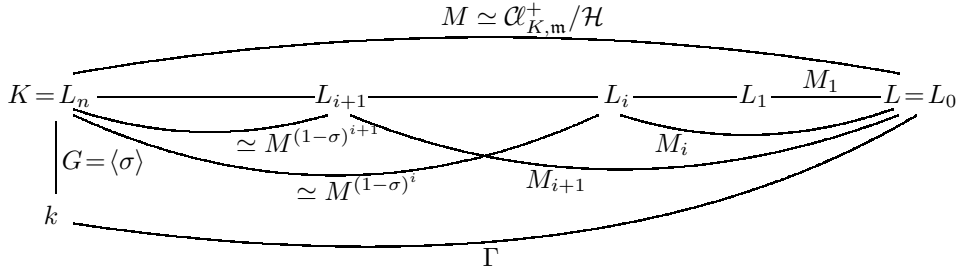
\unitlength=0.8cm
$$\vbox{\hbox{\hspace{-0.05cm} \vspace{-0.7cm} 
 \begin{picture}(9.5,5.1)
% horizontales
\put(7.3,3.50){\line(1,0){1.7}}
\put(9.6,3.50){\line(1,0){1.8}}
\put(9.05,3.4){$L_1$}
\put(10.1,3.65){$M_1$}
\put(-1.6,3.50){\line(1,0){3.6}}
\put(2.9,3.50){\line(1,0){3.8}}
% verticales
\put(-2.3,1.9){\line(0,1){1.20}}
\bezier{700}(-2.0,3.9)(5.0,5.1)(11.7,3.9)
\put(4.1,4.7){$M \simeq  \Cl_{K, {\mathfrak m}}^{+}/ {\mathcal H}$}
\bezier{700}(-2.0,1.4)(6.8,0.1)(12.0,3.2)
\put(4.8,0.7){$\Gamma$}
\bezier{400}(7.1,3.2)(9.3,2.5)(11.6,3.3)
\bezier{550}(2.6,3.2)(7.0,1.4)(11.7,3.2)
\put(7.65,2.6){$M_i$}
\bezier{400}(-2.0,3.2)(2.35,1.2)(6.7,3.2)
\put(1.7,1.8){$\simeq M^{(1-\sigma)^i}$}
\bezier{400}(-2.0,3.3)(0.1,2.6)(2.2,3.2)
\put(0.7,2.6){$\simeq M^{(1-\sigma)^{i+1}}$}
\put(11.45,3.4){$L\!=\!L_0$}
\put(2.0,3.4){$L_{i+1}$}
\put(6.8,3.4){$L_i$}
\put(-3.1,3.4){$K\!=\!L_n$}
\put(5.5,2.0){$M_{i+1}$}
\put(-2.5,1.4){$k$}
\put(-2.2,2.3){$G\!=\! \langle \sigma \rangle$}
\end{picture}   }} $$
\unitlength=1.0cm

\bigskip
In a dual manner, we have the following tower of fields where
$L'_i$ is the subfield of $L$ fixed by $M^{(1-\sigma)^i}$, whence
$[L'_i : K] = \#M_i$:

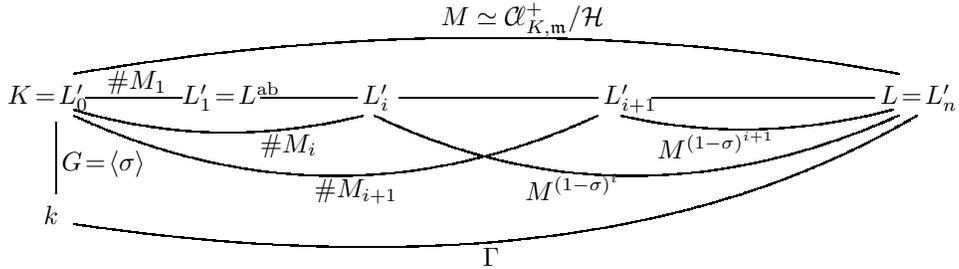
 \unitlength=0.8cm
$$\vbox{\hbox{\hspace{-0.05cm} \vspace{-0.7cm} 
 \begin{picture}(9.5,5.1)
% horizontales
\put(7.6,3.50){\line(1,0){3.7}}
\put(-1.8,3.50){\line(1,0){1.6}}
\put(1.1,3.50){\line(1,0){1.6}}
\put(3.4,3.50){\line(1,0){3.3}}
\put(-1.45,3.65){$\#M_1$}
% verticales
\put(-2.3,1.9){\line(0,1){1.20}}
\bezier{700}(-2.0,3.9)(5.0,5.1)(11.7,3.9)
\put(4.1,4.7){$M \simeq \Cl_{K, {\mathfrak m}}^{+}/ {\mathcal H}$}
\bezier{700}(-2.0,1.4)(6.8,0.1)(12.0,3.2)
\put(4.8,0.7){$\Gamma$}
\bezier{400}(7.1,3.2)(9.3,2.7)(11.6,3.3)
\bezier{550}(3.0,3.2)(7.2,1.2)(11.7,3.2)
\put(7.7,2.5){$M^{(1-\sigma)^{i+1}}$}
\bezier{400}(-2.0,3.2)(2.35,1.2)(6.7,3.2)
\bezier{400}(-2.0,3.3)(0.35,2.6)(2.8,3.2)
\put(1.1,2.6){$\#M_i$}
\put(2.0,1.85){$\#M_{i+1}$}
\put(11.4,3.4){$L\!=\!L'_n$}
\put(2.8,3.4){$L'_i$}
\put(-0.2,3.4){$L'_{1}\!=\! L^{\rm ab}$}
\put(6.8,3.4){$L'_{i+1}$}
\put(-3.1,3.4){$K\!=\!L'_0$}
\put(5.5,1.8){$M^{(1-\sigma)^i}$}
\put(-2.5,1.4){$k$}
\put(-2.2,2.3){$G\!=\! \langle \sigma \rangle$}
\end{picture}   }} $$
\unitlength=1.0cm

\medskip
Our method to compute $\#(M_{i+1}/M_{i})$ differs from classical ones by ``translating'' the 
well-known Chevalley's formula giving the number of ambiguous classes, see \eqref{eq28}, 
Remark \ref{rema3}), by means of the exact sequence of Theorem \ref{prop1} applied to a 
suitable ${\mathcal H} = {\mathcal H}_0$.

\smallskip
The main application is the case where $G$ is cyclic of order a prime $p$ and when $L/K$
is an Abelian finite $p$-extension defined via class field theory (e.g., various $p$-Hilbert class fields
in most classical practices). So, when the $M_i$ are computed, it is possible to give, 
under some assumptions (like $M^{1+\sigma + \cdots +\sigma^{p-1}} = 1$ and/or $\# M^G = p$), the 
structure of ${\rm Gal}(L/K)$ as $\Z_p[G]$-module or at least as Abelian $p$-group.

\smallskip
In the above example, this will give for instance the structure of the $p$-class group 
in the restricted sense from the knowledge of the $p$-class group of $k$ 
and some local normic computations in $K/k$.

\begin{remarks}
(i) In some french papers, we find the terminology {\it sens restreint {\rm vs} sens ordinaire}
which was introduced by J. Herbrand in \cite[VII, \S4]{H}, and we have used in \cite{Gr1} the
upperscripts ${}^{\rm res}$ and ${}^{\rm ord}$ to specify the sense; to be consistent with 
many of today's publications, we shall use here the words {\it narrow sense} instead 
of {\it restricted sense} and use the upperscript ${}^{+}$. However, we utilize 
$S$-objects, where $S$ is a suitable set of places ($S$-units, $S$-class groups,
$S$-class fields, etc.), so that $S=\emptyset$ corresponds to the restricted sense 
and totally positive elements~; the ordinary (or wide) sense corresponds to the choice 
of the set $S$ of real infinite places of the field, thus, for the ordinary sense, we must keep 
the upperscript ${}^{\rm ord}$ (see \S\S\ref{maindef}, \ref{clg}).

\smallskip
We shall consider generalized $S$-class groups modulo ${\mathfrak m}$ since any 
situation is available by choosing suitable ${\mathfrak m}$ and $S$
(including the case $p=2$ with ordinary and narrow senses).

\smallskip
(ii) It is clear that the study of $p$-class groups in $p$-extensions $K/k$ is rather easy compared to 
the ``semi-simple'' case (i.e., when $p \nmid {\rm Gal}(K/k)$); see, e.g., an overview in 
\cite{St1}, and an extensive algebraic study in \cite{L1} via representation theory, then in \cite{Ku}, 
\cite{Sch1}, \cite{Sch2}, \cite{Sch3}, \cite{SW}, and in \cite{Wa} for cyclotomic fields.

\smallskip
Indeed, the semi-simple case is of a more Diophantine framework and is part of an analytic setting 
leading to difficult well-known questions in Iwasawa theory \cite{Iw}, then in $p$-adic L-functions that we had 
conjectured in \cite[(1977)]{Gr14}, and which were initiated with the Thaine--Ribet--Mazur--Wiles--Kolyvagin 
``principal theorem'' \cite[(1984)]{MW} with significant developments by C. Greither and R. Ku\v cera
(e.g., \cite{GK1}, \cite{GK2}, \cite{GK3}, \cite{GK4}), which have in general no connection with the present text, 
part of the so called ``genera theory'' (except for the method of Section \ref{sect5} in which we 
obtain informations on the semi-simple case).
\end{remarks}

\section{Class field theory -- Generalized ideal class groups}\label{sect2}
We use, for some technical aspects, the principles defined in \cite{Gr2}; one can also use 
the works of Jaulent as \cite{Ja1}, \cite{Ja2}, of the same kind. For instance, for a real infinite place 
which becomes complex in an extension, we speak of {\it complexification} instead of {\it ramification}, 
and the corresponding {\it inertia} subgroup of order $2$ is called the {\it decomposition group} 
of the place; in other words this place has a {\it residue degree~$2$} instead of a {\it ramification index $2$}.
If the real place remains real by extension, we say as usual that this place splits 
(of course into two real places above) and that its residue degree is $1$. The great advantage is that
the moduli ${\mathfrak m}$ of class field theory are ordinary integer ideals, any situation being 
obtained from the choice of $S$. 

\smallskip
A consequence of this viewpoint is that the pivotal notion is the narrow sense.

\subsection{Numbers -- Ideals -- Ideal classes}\label{maindef}
Let $F$ be any number field (this will apply to $K$ and $k$). We denote by:

\medskip
(i)  $\Pl_F= \Pl_{F,0} \cup \Pl_{F,\infty}$, the set of finite and infinite places of $F$. The places
 (finite or infinite) are given as symbols ${\mathfrak p}$; the finite places are the prime ideals; 
the infinite places may be real or complex and are associated with the $r_1+r_2$ embeddings of $F$
 into $\R$ and $\C$ as usual (with $r_1+2\,r_2= [F : \Q]$);

\medskip
(ii)  $T \ \ \& \ \ S$, two disjoint sets of places of  $F$. We suppose that $T$ has only finite places 
and that $S=: S_0 \cup S_\infty$, $S_0 \subset \Pl_{F,0}$,  $S_ \infty \subset \Pl_{F, \infty}$,
where $S_ \infty$ does not contain any complex place;

\medskip
(iii) ${\mathfrak m}$, a modulus of $F$ with support $T$ (i.e., a nonzero integral ideal of $F$ divisible 
by each of the prime ideals ${\mathfrak p} \in T$ and not by any ${\mathfrak p} \notin T$);

\medskip
(iv) $v_{\mathfrak p}  : F^\times \to \Z$ is the normalized ${\mathfrak p}$-adic valuation 
when ${\mathfrak p}$ is a prime ideal; if ${\mathfrak p}$ is a real infinite place, then
$v_{\mathfrak p}  : F^\times \to \Z/2\,\Z$ is defined by $v_{\mathfrak p}(x)=0$
(resp. $v_{\mathfrak p}(x)=1$) if $\sigma_{\mathfrak p}(x)>0$ (resp. $\sigma_{\mathfrak p}(x)<0$)
where $\sigma_{\mathfrak p}$ is the corresponding embedding $F \to \R$ associated with 
${\mathfrak p}$; if ${\mathfrak p}$ is complex (thus corresponding to a pair of conjugated embeddings
$F \to \C$), then $v_{\mathfrak p}=0$.

\medskip
(v)  $\ \ F^{\times +} = \{x \in F^\times, \ \   v_{\mathfrak p} (x)=0, \ \, 
\forall {\mathfrak p} \in \Pl_{F, \infty} \}$, group of totally positive elements;

\smallskip
\hspace{0.75cm}
$U_{F, T} = \{x \in F^\times , \   v_{\mathfrak p}(x)=0, \  \forall  {\mathfrak p} \in T \}$;
$U_{F, T}^+ = U_{F, T} \cap F^{\times +}$;
 
\smallskip
\hspace{0.8cm}$U_{F, {\mathfrak m}} = \{x \in U_{F, T} , \ \, x \equiv 1 \pmod {\mathfrak m}\}$;
$\ U_{F, {\mathfrak m}}^+ = U_{F, {\mathfrak m}} \cap F^{\times +}$;

\medskip
(vi) $\ E_F^S = \{x \in F^\times , \ \,  v_{\mathfrak p}(x)=0, \ \forall  {\mathfrak p} \notin S \}$,
 group of $S$-units of $F$;

\smallskip
\hspace{0.7cm}$E_{F,{\mathfrak m}}^S = \{x \in E_F^S, \,\  x \equiv 1 \pmod {\mathfrak m}\}$;

\smallskip
\hspace{0.7cm}$E_{F,{\mathfrak m}}^{\Pl_\infty} =: E_{F,{\mathfrak m}}^{\rm ord}$, 
group of units (in the ordinary sense) $\varepsilon \equiv 1 \pmod {\mathfrak m}$;

\smallskip
\hspace{0.7cm}$E_{F,{\mathfrak m}}^\emptyset =: E_{F,{\mathfrak m}}^+$, group of totally positive units 
$\varepsilon \equiv 1 \pmod {\mathfrak m}$;

\medskip
(vii) $\ I_F$, group of fractional ideals of $F$;

\smallskip
\hspace{0.74cm}$P_F$, group of principal ideals $(x)$, $x \in F^\times$ (ordinary sense);

\smallskip
\hspace{0.74cm}$P_F^+$, group of principal ideals $(x)$, $x \in F^{\times +}$ (narrow sense);

\smallskip
\hspace{0.74cm}$I_{F, T}= \{ {\mathfrak a} \in I_F, \ \, v_{\mathfrak p}({\mathfrak a}) = 0, \ \forall {\mathfrak p}\in T \}$;
$P_{F, T}= P_F \cap I_{F, T}$; $P_{F, T}^+= P_F ^+\cap I_{F, T}$;

\smallskip
\hspace{0.72cm}$P_{F,{\mathfrak m}}= \{ (x), \ \,  x \in U_{F, {\mathfrak m}} \}$,
ray group modulo ${\mathfrak m}$ in the ordinary sense; 

\smallskip
\hspace{0.72cm}$P_{F,{\mathfrak m}}^+= \{ (x), \ \,  x \in U_{F, {\mathfrak m}}^+ \}$,
ray group modulo ${\mathfrak m}$ in the narrow sense; 

\medskip
(viii) $\, \Cl_{F,{\mathfrak m}}^{\rm ord} = I_{F, T} / P_{F,{\mathfrak m}}$, generalized ray class group modulo 
${\mathfrak m}$ (ordinary sense);

\smallskip
\hspace{0.8cm}$\Cl_{F,{\mathfrak m}}^{+} = I_{F, T} / P_{F,{\mathfrak m}}^+$, 
generalized ray class group modulo ${\mathfrak m}$ (narrow sense);

\smallskip
\hspace{0.8cm}$\Cl_{F,{\mathfrak m}}^S := \Cl_{F,{\mathfrak m}}^{+}/ \langle \cl (S) \rangle_\Z$,
$S$-class group modulo ${\mathfrak m}$ 
where $\langle \cl (S) \rangle_\Z$ is the subgroup of $\Cl_{F,{\mathfrak m}}^{+}$ generated
by the classes of ${\mathfrak p} \in S_0$ and, for real ${\mathfrak p}\in S_\infty$, by the classes
of the principal ideals $(x_{\mathfrak p}^{\mathfrak m})$ where the $x_{\mathfrak p}^{\mathfrak m} \in F^\times$
satisfy to the following congruences and signatures:

\smallskip
\centerline{$x_{\mathfrak p}^{\mathfrak m} \equiv 1 \pmod {\mathfrak m}, \ \ 
\sigma_{\mathfrak p}(x_{\mathfrak p}^{\mathfrak m})<0 \ \  \& \ \ 
\sigma_{\mathfrak q}(x_{\mathfrak p}^{\mathfrak m})>0 \ \forall {\mathfrak q} \in 
\Pl_{F, \infty}\setminus \{{\mathfrak p}\}$;}

\medskip
we have $P_F = \langle\, (x_{\mathfrak p} )\,\rangle^{}_{{\mathfrak p} \in \Pl_{F,\infty}} \cdot P_F^+\ \ \ \& \ \ \  
P_{F,{\mathfrak m}} = \langle\, (x_{\mathfrak p}^{\mathfrak m})\, 
\rangle^{}_{{\mathfrak p} \in \Pl_{F,\infty}} \cdot P_{F,{\mathfrak m}}^+$.

\medskip
Taking $S= \emptyset $, then $S=\Pl_{F, \infty}$, we find again 

\medskip
\centerline{$\Cl_{F,{\mathfrak m}}^\emptyset = \Cl_{F,{\mathfrak m}}^{+}$, then
$\Cl_{F,{\mathfrak m}}^{\Pl_{F, \infty}} = \Cl_{F,{\mathfrak m}}^{+} /
\cl(\langle\, (x_{\mathfrak p}^{\mathfrak m})\, \rangle^{}_{{\mathfrak p}\in\Pl_{F,\infty}} ) 
 = \Cl_{F,{\mathfrak m}}^{\rm ord}$.}
 
\medskip
(ix) $\cl_F : I_{F, T} \too \Cl_{F,{\mathfrak m}}^S$, canonical map which must be read as 
$\cl_{F,{\mathfrak m}}^S$ for suitable ${\mathfrak m}$ and $S$, according to the case of 
class group considered, when there is no ambiguity.

\subsection{Class fields and corresponding class groups} \label{clg}
We define the generalized Hilbert class fields as follows:

\medskip
(i) $H_F^{+}$ is the Hilbert class field in the narrow sense (maximal Abelian extension of $F$
unramified for prime ideals and possibly complexified at $\infty$, which
means that the field $H_F^{+}$ may be non-real even if $F$ is totally real); we have 

\medskip
\centerline{${\rm Gal}(H_F^{+}/F) \simeq \Cl_F^{+} = I_F / P_F^+$; }

\medskip
(ii) $\ H_F^{\Pl_\infty} = H_F^{\rm ord} \subseteq H_F^{+}$ is the Hilbert class field in the ordinary 
sense (maximal Abelian extension of $F$, unramified for prime ideals, and splitted at $\infty$); we have 

\medskip
\centerline{${\rm Gal}(H_F^{\rm ord}/F) \simeq \Cl_F^{\rm ord} = I_F / P_F$;}

\medskip
(iii) $\ H_F^S \subseteq H_F^{+}$ is the $S$-split Hilbert class field (maximal Abelian extension of $F$ 
unramified for prime ideals and splitted at $S$); we have 

\medskip
\centerline{${\rm Gal}(H_F^S/F) \simeq \Cl_F^S = \Cl_F^{+}/ \langle \cl_F (S) \rangle_\Z$; }

\medskip
recall that the decomposition group of ${\mathfrak p}\in S_0$ (resp. $S_\infty$) is given, in 
$\Cl_{F}^{+}$, by the cyclic group generated by the class of ${\mathfrak p}$ (resp. $(x_{\mathfrak p})$);
hence ${\rm Gal}(H_{F}^{+}/H_F^S)$, generated by these decomposition groups, is isomorphic to 
$ \langle \cl_F (S) \rangle_\Z$.

\medskip
(iv) $\ H_{F, {\mathfrak m}}^{+}$ is the ${\mathfrak m}$-ray class field in the narrow sense, 

\medskip
\hspace{0.75cm}$H_{F, {\mathfrak m}}^{\rm ord}$ is the ${\mathfrak m}$-ray class field in the ordinary sense,

\medskip
\hspace{0.75cm}$H_{F, {\mathfrak m}}^S$ is the $S$-split ${\mathfrak m}$-ray class field of $F$
(denoted $F{\scriptstyle (\mathfrak m)}^S$ in \cite{Gr2}); we have

\medskip
\centerline{${\rm Gal}(H_{F, {\mathfrak m}}^S/F) \simeq \Cl_{F, {\mathfrak m}}^S
= \Cl_{F, {\mathfrak m}}^{+} / \langle \cl_F (S) \rangle_\Z$}

\medskip
(see (viii) and (ix) for the suitable definitions of $\cl_F$ depending on the class group considered).
In other words, $H_{F, {\mathfrak m}}^S$
is the maximal subextension of $H_{F, {\mathfrak m}}^{+}$ in which the (finite and infinite) places of
$S$ are totally split. 

\smallskip
For instance, for a prime $p$, the $p$-Sylow subgroups of $\Cl_F^{\rm ord}$ and
$\Cl_F^{\Pl_p}$, for the set $S = \Pl_p := \{{\mathfrak p}, \ \  {\mathfrak p} \mid p\}$, 
have a significant meaning in some duality theorems.

\section{Computation for the order of $\big(\Cl_{K,{\mathfrak m}}^{+} / {\mathcal H} \big)^G$}\label{sect3}
Let $K/k$ be any cyclic extension of number fields, of degree $d$, of Galois group $G$, and let
$\sigma$ be a fixed generator of $G$. We fix a modulus ${\mathfrak m}$ of $k$ with support $T$ 
which implies that $H_{K,{\mathfrak m}}^{+}/k$ is Galois (by
abuse we keep the same notation for the extensions of ${\mathfrak m}$ and $T$ in $K$). Then let 
$${\mathcal H} \subseteq \Cl_{K,{\mathfrak m}}^{+}$$ 

be an arbitrary sub-$G$-module of $\Cl_{K,{\mathfrak m}}^{+}$.

\begin{remarks} \label{rem1}{\rm 
(i) The group $G$ acts on $\Cl_{K,{\mathfrak m}}^{+}$, hence on 
${\rm Gal}(H_{K,{\mathfrak m}}^{+}/K)$ by conjugation via the Artin isomorphism
${\mathfrak A} \mapsto \big( \frac{H_{K,{\mathfrak m}}^{+}/K}{{\mathfrak A}} \big)
\in {\rm Gal}(H_{K,{\mathfrak m}}^{+}/K)$,
for all ${\mathfrak A} \in I_{K,T}$ (modulo $P_{K,{\mathfrak m}}^+$), for which

\smallskip
\centerline{$ \Big( \frac{H_{K,{\mathfrak m}}^{+}/K}{{\mathfrak A^\tau}} \Big) = 
\tau \cdot \Big( \frac{H_{K,{\mathfrak m}}^{+}/K}{{\mathfrak A}} \Big)\cdot \tau^{-1}$,
for all $\tau \in G$.}

\medskip
(ii) The sub-$G$-module ${\mathcal H}$ fixes a field $L \subseteq H_{K, {\mathfrak m}}^{+}$
which is Galois over $k$ and in the same way, ${\rm Gal}(L/K) 
\simeq \Cl_{K, {\mathfrak m}}^{+} / {\mathcal H}$ is a $G$-module.

\smallskip
(iii) Taking ${\mathcal H} =  \langle \cl_K (S) \rangle_\Z$, $S \subset \Pl_K$ (see (viii))
leads to $\Cl_{K,{\mathfrak m}}^{+} / {\mathcal H} = \Cl_{K,{\mathfrak m}}^S
\ \  \& \ \  L=H_{K, {\mathfrak m}}^S$ (assuming that $\cl_K(S)$ is a sub-$G$-module).

\medskip
(iv)  If we take, more generally, a modulus ${\mathfrak M}$ of $K$ ``above ${\mathfrak m}$'', 
it must be invariant by $G$; so necessarily, ${\mathfrak M}=({\mathfrak m})$ extended to $K$, 
except if some ${\mathfrak P} \mid {\mathfrak M}$ is ramified since $({\mathfrak p}) = 
\prod_{{\mathfrak P} \mid {\mathfrak p}} {\mathfrak P}^{e_{\mathfrak p}}$.
But in class field theory, it is always possible to work with a multiple ${\mathfrak M}'$ of 
${\mathfrak M}$ (because $H_{K, {\mathfrak M}}^{+} \subseteq H_{K, {\mathfrak M}'}^{+}$), 
so that the case ${\mathfrak M}=({\mathfrak m})$ is universal for our purpose and is, in practice,
any multiple of the conductor ${\mathfrak f}_{L/K}$ of $L/K$. }
\end{remarks}

We intend to compute $\# (\Cl_{K,{\mathfrak m}}^{+} / {\mathcal H})^G = \# {\rm Gal}(L/K)^G$,
which is equivalent, from exact sequences \eqref{eq1}, to obtain the degree $[L^{\rm ab} : K]$, 
where $L^{\rm ab}$ is the maximal subextension of $L$, Abelian over $k$. 

\smallskip
Our method is straightforward and is based on the well-known ``ambiguous class number formula''
given by Chevalley \cite[(1933)]{Ch1}, and used in any work on class field theory
(e.g.,   \cite{Ch2}, \cite{AT}, \cite{L}, \cite[Chap. 3]{Ja1}, \cite{L3}), often in a hidden manner,
since it is absolutely necessary for the interpretation, in the cyclic case, of the famous idelic index
$(J_k : k^\times {\rm N}_{K/k} (J_K)) = [K^{\rm ab} : k]$, valid for any finite extension $K/k$ 
and which gives the {\it product formula} between normic symbols
in view of the Hasse norm theorem (in the cyclic case).

\smallskip
This formula has also some importance for Greenberg's conjectures \cite{Gre2} on Iwasawa's 
$\lambda, \mu$  invariants for the $\Z_p$-extensions of a totally real number field
\cite{Gr15}.

\smallskip
Chevalley's formula in the cyclic case is based on (and roughly speaking equivalent to) the nontrivial 
computation of the Herbrand quotient 
$\frac{(E_k : {\rm N}_{K/k}(E_K))} {({}_{\rm N}E_K : E_K^{1-\sigma})} = \frac{2^{\rm rc}}{[K:k]}$
of the group of units $E_K$, where ${}_{\rm N} E_K$ is the subgroup of units of norm $1$
in $K/k$ and where ${\rm rc}$ is the number of real places of $k$, complexified in $K$.
Chevalley's formula was established first by Takagi for cyclic extensions of prime degree $p$; 
the generalization to arbitrary cyclic case by Chevalley was possible due to the so called 
``Herbrand theorem on units'' \cite{H}.

\smallskip
Many fixed point formulas where given in the same framwork for other notions of classes 
(e.g., logarithmic class groups, \cite{Ja3}, \cite{So}, $p$-ramification torsion groups, 
\cite[Theorem IV.3.3]{Gr2}, \cite{MoNg}).

\subsection{The main exact sequence and the computation of 
$\#(\Cl_{K,{\mathfrak m}}^{+}/{\mathcal H})^G$}\label{ssect}
\subsubsection{Global computations}
Recall that ${\mathcal H}$ is a sub-$G$-module of 
$\Cl_{K,{\mathfrak m}}^{+}\! =\! I_{K,T}/ P_{K,{\mathfrak m}}^+$. Put 
\begin{equation}\label{eq2}
\wt {\mathcal H} = \{h \in \Cl_{K,{\mathfrak m}}^{+}, \ \, h^{1-\sigma} \in {\mathcal H}\} ; 
\end{equation}

it is obvious that
\begin{equation}\label{eq3}
\big (\Cl_{K,{\mathfrak m}}^{+} / {\mathcal H} \big)^G = \wt {\mathcal H} / {\mathcal H}. 
\end{equation}

We have the exact sequences
\begin{equation}\label{eq4}
\begin{aligned}
1\too \Cl_{K,{\mathfrak m}}^{{+}\,G} \tooo \wt {\mathcal H} 
\mathop{\tooo}^{1-\sigma} (\wt {\mathcal H})^{1-\sigma} \too 1 \\ 
1\too {}_{{\rm N}}{\mathcal H} \tooo  {\mathcal H} \mathop{\tooo}^{{\rm N}_{K/k}} {\rm N}_{K/k} ({\mathcal H}) \too 1, 
\end{aligned}
\end{equation}

with ${}_{{\rm N}}{\mathcal H} = {\rm Ker}( {\rm N}_{K/k})$, where ${\rm N}_{K/k}$ denotes the 
{\it arithmetical norm}\,\footnote{\,For $K/k$ Galois,
the arithmetical norm ${\rm N}_{K/k}$ is defined
multiplicatively on the group of ideals of $K$ by ${\rm N}_{K/k} ({\mathfrak P}) =
{\mathfrak p}^{f_{\mathfrak p}}$ for prime ideals ${\mathfrak P}$ of $K$, where
${\mathfrak p}$ is the prime ideal of $k$ under ${\mathfrak P}$ and
$f_{\mathfrak p}$ its residue degree in $K/k$. 
If ${\mathfrak A}=(\alpha)$ is principal in $K$, then 
${\rm N}_{K/k} ({\mathfrak A}) = ({\rm N}_{K/k}(\alpha))$ in $k$.}
as opposed to the {\it algebraic norm} defined in $\Z[G]$
by $\nu_{K/k} = 1+ \sigma + \cdots + \sigma^{d-1}$, and for which we have the relation
$\nu_{K/k} = j_{K/k} \circ {\rm N}_{K/k}$, where $j_{K/k}$ is the map of extension of ideals from $k$ to $K$
(it corresponds, via the Artin map, to the transfer map for Galois groups);
for a prime ideal ${\mathfrak P}$ of $K$, $j_{K/k} \circ {\rm N}_{K/k}({\mathfrak P}) = 
j_{K/k}({\mathfrak p}^{f_{\mathfrak p}})
= (\prod_{{\mathfrak P'} \mid {\mathfrak p}} {\mathfrak P'}^{e_{\mathfrak p}})^{f_{\mathfrak p}}$
(where $e_{\mathfrak p}$ is the ramification index),
which is indeed ${\mathfrak P}^{\nu_{K/k}}$ since $G$ operates transitively on the
${\mathfrak P'}\! \mid {\mathfrak p}$ with a decomposition group of order 
$\frac{[K : k]}{e_{\mathfrak p}f_{\mathfrak p}}$.

\smallskip
By definition, for an ideal ${\mathfrak A}$ of $K$, we have ${\rm N}_{K/k}(\cl_K ({\mathfrak A}) )
= \cl_k ({\rm N}_{K/k}({\mathfrak A}))$, and for any ideal ${\mathfrak a}$ of $k$, we have
$j_{K/k}(\cl_k ({\mathfrak a}) ) = \cl_K  (j_{K/k}({\mathfrak a}))$,
which makes sense since ${\rm N}_{K/k}(P_{K ,{\mathfrak m}}^+) \subseteq 
P_{k ,{\mathfrak m}}^+$ and $j_{K/k}(P_{k ,{\mathfrak m}}^+) \subseteq 
P_{K ,{\mathfrak m}}^+$, seeing the modulus ${\mathfrak m}$ of $k$ extended in 
$K$ in some writings. 

\smallskip
To simplify the formulas, we write ${\rm N}$ for ${\rm N}_{K/k}$.

\medskip
Recall that for $\ell$ prime, such that $\ell \nmid d=[K : k]$, the $\ell$-Sylow subgroup
$\Cl_{k,{\mathfrak m}}^{+} \otimes \Z_\ell$ is isomorphic to
$(\Cl_{K,{\mathfrak m}}^{+} \otimes \Z_\ell)^G$ since the map
$j_{K/k} : \Cl_{k,{\mathfrak m}}^{+} \otimes \Z_\ell \too 
\Cl_{K,{\mathfrak m}}^{+} \otimes \Z_\ell$ is injective,
and the map ${\rm N}_{K/k} : \Cl_{K,{\mathfrak m}}^{+} \otimes \Z_\ell \too 
\Cl_{k,{\mathfrak m}}^{+} \otimes \Z_\ell$ is surjective.

\medskip
Let ${\mathcal I}$ be any {\it subgroup} of $I_{K, T}$ such that $\cl_K ({\mathcal I}) = {\mathcal H}$, i.e.,
\begin{equation}\label{eq5}
{\mathcal I}\cdot P_{K ,{\mathfrak m}}^+\, \big /P_{K ,{\mathfrak m}}^+ = {\mathcal H} ;
\end{equation}

the group ${\mathcal I}\cdot P_{K ,{\mathfrak m}}^+$ is unique and we then have 
\begin{equation}\label{eq6}
{\rm N}({\mathcal H}) =  {\rm N}({\mathcal I}) \cdot P_{k ,{\mathfrak m}}^+\, \big /
P_{k ,{\mathfrak m}}^+ \simeq {\rm N}({\mathcal I}) \big /
 {\rm N}({\mathcal I}) \cap P_{k ,{\mathfrak m}}^+ .
\end{equation}

\begin{remark} \label{IS}
The generalized class groups being finite and since any ideal class can be represented 
by a finite or infinite place, we can find  a finite set $S_K = S_{K,0} \cup S_{K,\infty}$ of non-complex
places such that the classes ${\mathfrak P} \cdot P_{K ,{\mathfrak m}}^+$ (for ${\mathfrak P} \in S_{K,0}$)
and $(x_{\mathfrak P}^{\mathfrak m}) \cdot P_{K ,{\mathfrak m}}^+$ (for ${\mathfrak P} \in S_{K,\infty}$)
generate ${\mathcal H}$, so that we can take ${\mathcal I} = \plus_{{\mathfrak P} \in S_{K,0}}
\langle {\mathfrak P} \rangle_\Z \cdot  \plus_{{\mathfrak P} \in S_{K,\infty}}
\langle (x_{\mathfrak P}^{\mathfrak m})\rangle_\Z$ as canonical subgroup 
of $I_{K, T}$ defining ${\mathcal H}$.
Thus $\Cl_{K,{\mathfrak m}}^{+} / {\mathcal H}=
\Cl_{K,{\mathfrak m}}^{S_K}$ in the meaning of \S\,\ref{maindef} (viii). But, to ease the forthcoming computations,
we keep the writing with the subgroup ${\mathcal I}$. 

\smallskip
Note that we do not assume that ${\mathcal I}$ or 
$S_K$ are invariant under $G$ contrary to ${\mathcal H}$ and ${\mathcal I}\cdot P_{K ,{\mathfrak m}}^+$;
so, if for instance ${\mathfrak P} \in S_{K,0}$, for any $\tau \in G$ we have, 
$\cl_K({\mathfrak P}^\tau) = \cl_K({\mathfrak P}')$ for some ${\mathfrak P} '\in S_K$, whence 
${\mathfrak P}^\tau = {\mathfrak P}' \,(x)$, $x \in U_{K ,{\mathfrak m}}^+$.
\end{remark}

From the exact sequence, where $\psi(u) = (u)$ for all $u \in U_{k ,{\mathfrak m}}^+$,
\begin{equation}\label{eq7}
1 \too E_{k ,{\mathfrak m}}^+ \tooo U_{k ,{\mathfrak m}}^+ \mathop{\tooo}^\psi P_{k ,{\mathfrak m}}^+
\too 1,
\end{equation}

we then put
\begin{equation}\label{eq8}
\Lambda := \psi^{-1} \big(  {\rm N}({\mathcal I}) \cap P_{k ,{\mathfrak m}}^+ \big) =
\{x \in U_{k ,{\mathfrak m}}^+, \  (x) \in  {\rm N}({\mathcal I}) \};
\end{equation}

we have the obvious inclusions $E_{k ,{\mathfrak m}}^+ \subseteq 
\Lambda \subseteq U_{k ,{\mathfrak m}}^+$. 

\smallskip
We can state (fundamental exact sequence):

\begin{theorem}\label{prop1}
Let $K/k$ be any cyclic extension, of Galois group $G =: \langle \sigma \rangle$.
Let ${\mathcal H} = {\mathcal I}\cdot P_{K ,{\mathfrak m}}^+\, \big /P_{K ,{\mathfrak m}}^+$ 
be a sub-$G$-module of $\Cl_{K,{\mathfrak m}}^{+}$, where ${\mathcal I}$ is a
subgroup of $I_{K,T}$, and let
$$\wt {\mathcal H} = \{h \in \Cl_{K,{\mathfrak m}}^{+}, \ \, h^{1-\sigma} \in {\mathcal H}\}. $$

We have $(\wt {\mathcal H})^{1-\sigma} \subseteq  {}_{{\rm N}}{\mathcal H}$ 
and the exact sequence (see \eqref{eq2} and \eqref{eq5} to \eqref{eq8}):
\begin{equation}\label{eq9}
1\too \big( E_{k ,{\mathfrak m}}^+  {\rm N}(U_{K ,{\mathfrak m}}^+) \big) \cap \Lambda
\tooo \Lambda \mathop{\tooo}^\varphi  {}_{{\rm N}}{\mathcal H} \big / (\wt {\mathcal H})^{1-\sigma} \too 1,
\end{equation}

where, for all $x \in \Lambda$, $\varphi(x) = \cl_K ({\mathfrak A}) \!\cdot \! (\wt {\mathcal H})^{1-\sigma}\!$, 
for any ${\mathfrak A} \in {\mathcal I}$ such that ${\rm N}({\mathfrak A}) = (x)$.
\end{theorem}

\begin{proof} If $x \in \Lambda$, we have $(x) \in P_{k ,{\mathfrak m}}^+$
and by definition $(x)$ is of the form
${\rm N}({\mathfrak A})$, ${\mathfrak A} \in {\mathcal I}$, and thus 
$\cl_K ({\mathfrak A}) \in {}_{{\rm N}}{\mathcal H}$; if
$(x) = {\rm N}({\mathfrak B})$, ${\mathfrak B} \in {\mathcal I}$, there exists
${\mathfrak C} \in I_K$ such that ${\mathfrak B} \cdot {\mathfrak A}^{-1} = {\mathfrak C}^{1-\sigma}$.
It is known that $I_{K, T}$ is a $\Z[G]$-module (and a free $\Z$-module) such that 
${\rm H}^1(G, I_{K, T})=0$; since ${\mathfrak B} \cdot {\mathfrak A}^{-1} \in I_{K, T}$ 
is of norm $1$, it is of the required form with ${\mathfrak C} \in I_{K, T}$. Then 
$$(\cl_K ({\mathfrak C}))^{1-\sigma} =
\cl_K ({\mathfrak C}^{1-\sigma}) = 
\cl_K ({\mathfrak B} \cdot {\mathfrak A}^{-1}) \in
\cl_K ({\mathcal I}) = {\mathcal H}, $$
and by definition $\cl_K ({\mathfrak C}) \in \wt {\mathcal H}$, which
implies $(\cl_K ({\mathfrak C}))^{1-\sigma} \in (\wt {\mathcal H})^{1-\sigma}$.
Hence the fact that the map $\varphi$ is well defined.

\smallskip
If ${\mathfrak A} \in {\mathcal I}$ is such that $\cl_K({\mathfrak A}) \in  {}_{{\rm N}}{\mathcal H}$,
then ${\rm N}({\mathfrak A}) = (x)$, $x \in U_{k ,{\mathfrak m}}^+$, thus $x \in \Lambda$
and it is a preimage; hence the surjectivity of $\varphi$.

\smallskip
We now compute ${\rm Ker}(\varphi)$: if $x \in \Lambda$, $(x) = {\rm N}({\mathfrak A})$, 
${\mathfrak A} \in {\mathcal I}$, and if $\cl_K ({\mathfrak A}) \in 
(\wt {\mathcal H})^{1-\sigma}$, there exists ${\mathfrak B} \in I_{K,T}$ such that
$\cl_K ({\mathfrak B} )\in \wt {\mathcal H}$ and
$\cl_K ({\mathfrak A}) = \cl_K ({\mathfrak B})^{1-\sigma}$; 
so there exists $u \in U_{K ,{\mathfrak m}}^+$ such that ${\mathfrak A} = {\mathfrak B}^{1-\sigma} \cdot (u)$,
giving $(x) =  {\rm N}({\mathfrak A}) = ({\rm N}(u) )$, hence
$$x= \varepsilon \cdot {\rm N}(u),\ \  \varepsilon\in E_k^{\rm ord}; $$

since $x$ and ${\rm N}(u)$ are in $U_{k ,{\mathfrak m}}^+$, we get $\varepsilon\in E_{k,{\mathfrak m}}^+$
and $x \in E_{k,{\mathfrak m}}^+ \, {\rm N}(U_{K ,{\mathfrak m}}^+)$.

\smallskip
Reciprocally, if $x \in \Lambda$ is of the form $\varepsilon \cdot {\rm N}(u)$, 
$\varepsilon\in E_{k,{\mathfrak m}}^+$ and $u \in U_{K ,{\mathfrak m}}^+$, this yields
$$(x) = {\rm N}(u) =  {\rm N}({\mathfrak A}),\ \ {\mathfrak A} \in {\mathcal I}, $$ 

which leads to the relation
${\mathfrak A} = (u) \cdot {\mathfrak B}^{1-\sigma}$ where, as we know, we can 
choose ${\mathfrak B} \in  I_{K,T}$ since ${\mathfrak A} \, (u)^{-1} \in I_{K ,T}$.
Since $(u) \in P_{K ,{\mathfrak m}}^+$, $\cl_K ({\mathfrak B})^{1-\sigma}=
\cl_K ({\mathfrak A}) \in {\mathcal H}$, hence $ \cl_K ({\mathfrak B}) \in \wt {\mathcal H}$,
and we obtain $\cl_K ({\mathfrak A}) \in (\wt {\mathcal H})^{1-\sigma}$.
\end{proof}

We deduce from \eqref{eq4},
\begin{equation}\label{eq10}
(\wt  {\mathcal H} :   {\mathcal H}) = \ds
\frac{\# \Cl_{K ,{\mathfrak m}}^{{+}\,G} \cdot \# (\wt {\mathcal H})^{1-\sigma}}
{\# {\rm N}({\mathcal H}) \cdot \#  {}_{{\rm N}}{\mathcal H}} = \frac{\# \Cl_{K ,{\mathfrak m}}^{{+}\,G}}
{\# {\rm N}({\mathcal H}) \cdot  ( {}_{{\rm N}}{\mathcal H} : (\wt {\mathcal H})^{1-\sigma})} ;
\end{equation}

thus from \eqref{eq3}, \eqref{eq9} and \eqref{eq10},
\begin{equation}\label{eq11}
\begin{aligned}
\# \big(\Cl_{K ,{\mathfrak m}}^{+}/ {\mathcal H} \big)^G 
&= \frac{\# \Cl_{K ,{\mathfrak m}}^{{+}\,G}}{ \# {\rm N}({\mathcal H}) \cdot 
(\Lambda : (E_{k,{\mathfrak m}}^+\,{\rm N}(U_{K ,{\mathfrak m}}^+) )\cap \Lambda)}  \\
&= \frac{\# \Cl_{K ,{\mathfrak m}}^{{+}\, G}}{ \# {\rm N}({\mathcal H}) \cdot 
(\Lambda \,{\rm N}(U_{K,{\mathfrak m}}^+) : E_{k,{\mathfrak m}}^+\,
{\rm N}(U_{K,{\mathfrak m}}^+))}. 
\end{aligned}
\end{equation}

We first apply this formula to
$${\mathcal H}_0 =P_{K ,T}^+ / P_{K ,{\mathfrak m}}^+ \simeq
(U_{K,T}^+ / U_{K,{\mathfrak m}}^+) \big / (E_K^+/ E_{K,{\mathfrak m}}^+)$$

which is the sub-module of $ \Cl_{K ,{\mathfrak m}}^{+}$ corresponding to the Hilbert class 
field $H_K^{+}$ since, using the id\'elic Chinese remainder theorem (cf. \cite[Remark I.5.1.2]{Gr2}), 
or the well-known fact that any class contains a representative prime to $T$, we get the 
surjection $I_{K,T}/P_{K,T}^+ \to I_K/P_K^+$ giving an isomorphisme, whence
\begin{equation}\label{eq12}
\big(\Cl_{K ,{\mathfrak m}}^{+}/ {\mathcal H}_0 \big)^G \simeq (I_{K,T}/P_{K,T}^+)^G
\simeq (I_K/P_K^+)^G \simeq \Cl_K^{{+}\,G}.
\end{equation}

Take  ${\mathcal I}_0 := P_{K ,T}^+$; then
\begin{equation}\label{eq13}
{\rm N}({\mathcal I}_0)  ={\rm N}(P_{K ,T}^+) \ \ \& \ \  
{\rm N}({\mathcal H}_0)  = {\rm N}(P_{K ,T}^+)\cdot  P_{k ,{\mathfrak m}}^+/ P_{k ,{\mathfrak m}}^+,
\end{equation}

and 
\begin{equation}\label{eq14}
\Lambda_0 = \{x \in U_{k ,{\mathfrak m}}^+ , \ \, (x) \in {\rm N}(P_{K ,T}^+)\} 
= (E_k^+ \,  {\rm N}(U_{K ,T}^+)) \cap U_{k ,{\mathfrak m}}^+. 
\end{equation}

It follows, from \eqref{eq11} applied to ${\mathcal H}_0$, from \eqref{eq12}, and 
${\rm N} (U_{K ,{\mathfrak m}}^+) \subseteq \Lambda_0$ (see \eqref{eq14}),
\begin{equation}\label{eq15}
\# \Cl_{K ,{\mathfrak m}}^{{+}\,G} = \# \Cl_K^{{+}\,G}  \cdot \# {\rm N}({\mathcal H}_0) \cdot 
\big ( (E_k^+ \,  {\rm N}(U_{K ,T}^+)) \cap U_{k ,{\mathfrak m}}^+   :
E_{k ,{\mathfrak m}}^+ \, {\rm N} (U_{K ,{\mathfrak m}}^+) \big).
\end{equation}

Now, ${\rm N}({\mathcal H}_0)$ in \eqref{eq13} can be interpreted by means of the exact sequence
\begin{equation*}
\begin{aligned}
1\too  E_k^+ U_{k ,{\mathfrak m}}^+ \big / U_{k ,{\mathfrak m}}^+ \tooo  
E_k^+ {\rm N}(& U_{K ,T}^+)  U_{k ,{\mathfrak m}}^+ \big / U_{k ,{\mathfrak m}}^+  \\
& \tooo {\rm N}({\mathcal H}_0) = {\rm N}(P_{K ,T}^+)\cdot  P_{k ,{\mathfrak m}}^+ \big / P_{k ,{\mathfrak m}}^+ \too 1, 
\end{aligned}
\end{equation*}
 
giving
\begin{equation}\label{eq16}
\# {\rm N}({\mathcal H}_0) = \frac{(E_k^+ \, {\rm N}(U_{K ,T}^+) : 
(E_k^+ \, {\rm N}(U_{K ,T}^+) )\cap U_{k ,{\mathfrak m}}^+)}
{(E_k^+ : E_{k ,{\mathfrak m}}^+ )};
\end{equation}

thus from \eqref{eq15} and \eqref{eq16},
\begin{equation}\label{eq17}
\# \Cl_{K ,{\mathfrak m}}^{{+}\,G} = \# \Cl_K^{{+}\,G} \cdot
\frac{(E_k^+ \,   {\rm N}(U_{K ,T}^+) :  E_{k ,{\mathfrak m}}^+ \,  {\rm N}(U_{K ,{\mathfrak m}}^+) ) }
{(E_k^+ : E_{k ,{\mathfrak m}}^+ )}.
\end{equation}

The inclusions ${\rm N}(U_{K ,{\mathfrak m}}^+) \subseteq 
 E_{k ,{\mathfrak m}}^+ \,  {\rm N}(U_{K ,{\mathfrak m}}^+) \subseteq
E_k^+ \,   {\rm N}(U_{K ,T}^+)$ lead from \eqref{eq17} to
$$\# \Cl_{K ,{\mathfrak m}}^{{+}\,G} = \# \Cl_K^{{+}\,G} \cdot
\frac{( E_{k }^+ \,  {\rm N}(U_{K ,T}^+) : {\rm N}(U_{K ,{\mathfrak m}}^+) )}
{(E_k^+ : E_{k ,{\mathfrak m}}^+ ) \cdot (E_{k ,{\mathfrak m}}^+ \,  {\rm N}(U_{K ,{\mathfrak m}}^+) : 
{\rm N}(U_{K ,{\mathfrak m}}^+)) }, $$

in other words
\begin{equation}\label{eq18}
\# \Cl_{K ,{\mathfrak m}}^{{+}\,G} =  
\# \Cl_K^{{+}\,G} \cdot \frac{(E_k^+ \,{\rm N}(U_{K ,T}^+) :  {\rm N}(U_{K ,T}^+) )\cdot
({\rm N}(U_{K ,T}^+ ): {\rm N}(U_{K ,{\mathfrak m}}^+))}
{(E_k^+ : E_{k ,{\mathfrak m}}^+ ) \cdot  (E_{k ,{\mathfrak m}}^+ \,  {\rm N}(U_{K ,{\mathfrak m}}^+) : 
{\rm N}(U_{K ,{\mathfrak m}}^+))}.
\end{equation}

Chevalley's formula in the narrow sense (\cite[Lemma II.6.1.2]{Gr2}, \cite[p. 177]{Ja1}) is
\begin{equation}\label{eq19}
\# \Cl_K^{{+} \, G} = \frac{\# \Cl_k^{+}\cdot \prod_{{\mathfrak p} \in \Pl_{k,0}} e_{\mathfrak p}}
{[K : k] \cdot (E_k^+ : E_k^+ \cap {\rm N}(K^\times))},
\end{equation}

where $e_{\mathfrak p}$ is the ramification index in $K/k$ of the finite place ${\mathfrak p}$.

\begin{lemma}\label{lemnorm}  For any finite set $T$, we have the relation
\begin{equation}\label{eq20}
U_{k,T}^+ \cap {\rm N}(K^\times) = {\rm N}(U_{K,T}^+).
\end{equation}
\end{lemma}

\begin{proof} Let $x \in U_{k,T}^+$ of the form ${\rm N}(z)$, $z \in K^\times$; put
$(z) = \prod_{{\mathfrak p} \in \Pl_{k,0}} {\mathfrak C}_{\mathfrak p}$ where
${\mathfrak C}_{\mathfrak p} = {\mathfrak P}_0^\omega$, for a fixed
${\mathfrak P}_0 \mid  {\mathfrak p}$ and $\omega \in \Z[G]$ depending on ${\mathfrak p}$.
Since the ${\rm N}({\mathfrak C}_{\mathfrak p}) = {\rm N}({\mathfrak P}_0^\omega)$ 
must be prime to $T$, we have $\omega \in (1-\sigma)\cdot \omega'$, $\omega' \in  \Z[G]$,
for all ${\mathfrak p} \in T$.
Hence $(z) = {\mathfrak C}\cdot {\mathfrak A}^{1-\sigma}$ with ${\mathfrak C} \in I_{K,T}$
and $ {\mathfrak A} \in I_K$.
We can choose in the class
modulo $P_K^+$ (narrow sense) of ${\mathfrak A}$ an ideal ${\mathfrak B}$ prime to $T$, 
hence ${\mathfrak B} = {\mathfrak A}\cdot (y')$, $y' \in K^{\times +}$, giving $z' := z\,y'{}^{1-\sigma}$
prime to $T$; then we can multiply $y'$ by $y''$, prime to $T$, to obtain $y:= y'\,y''$ 
such that the signature of $y^{1-\sigma}$ be suitable, which is possible because of the relation 
${\rm N}(z) \gg 0$ (i.e., the signature of $z$ is in the kernel of the norm, see \cite[Proposition 1.1]{Gr3}); 
then $z'' := z\,y^{1-\sigma}$ yields ${\rm N}(z'') = x$ with $z'' \in U_{K,T}^+$.
\end{proof}

So $(E_k^+ : E_k^+ \cap {\rm N}(U_{K,T}^+)) = (E_k^+  : E_k^+ \cap {\rm N}(K^\times))$.
More generally, if $x \in U_{k,T}^+$ must be in ${\rm N}(U_{K,T}^+)$ this is equivalent to
say that $x$ must be in ${\rm N}(K^\times)$ (i.e., a global norm without any supplementary 
condition) which is more convenient to use normic criteria (with Hasse's symbols 
$\big(\frac{x\,,\,{K}/k}{{\mathfrak p}} \big)$ for instance; see Remark \ref{nrs}). 
Recall that for $T=\emptyset$, $U_{K,T}^+=K^{\times +}$ and the lemma says that
$k^{\times +} \cap {\rm N}(K^\times) = {\rm N}(K^{\times +})$.

\smallskip
The lemma is valid with a modulus ${\mathfrak m}$ if its support $T$ has no ramified places.

\smallskip
From \eqref{eq18} , \eqref{eq19} and \eqref{eq20}, we have obtained
$$\# \Cl_{K ,{\mathfrak m}}^{{+}\, G} = \frac{\# \Cl_k^{+}\cdot 
\prod_{{\mathfrak p} \in \Pl_{k,0}} e_{\mathfrak p} \cdot ({\rm N}(U_{K ,T}^+ ): 
 {\rm N}(U_{K , {\mathfrak m}}^+))} {[K : k] \cdot (E_k^+ : E_{k, {\mathfrak m}}^+)  
\cdot (E_{k, {\mathfrak m}}^+ \, {\rm N}(U_{K , {\mathfrak m}}^+) : 
{\rm N}(U_{K,{\mathfrak m}}^+))}, $$

hence using \eqref{eq11}
\begin{equation}\label{eq21}
\#\big( \Cl_{K ,{\mathfrak m}}^{+} /{\mathcal H}\big)^G=  \frac{\# \Cl_k^{+}\cdot 
\prod_{{\mathfrak p} \in \Pl_{k,0}} e_{\mathfrak p} \cdot ( {\rm N}(U_{K ,T}^+) : 
 {\rm N}(U_{K , {\mathfrak m}}^+))} {[K : k] \cdot \# {\rm N}({\mathcal H})  \cdot (E_k^+ : E_{k, {\mathfrak m}}^+) 
\cdot (\Lambda \, {\rm N}(U_{K,{\mathfrak m}}^+) : {\rm N}(U_{K,{\mathfrak m}}^+))} .
\end{equation}

\subsubsection{Local study of $({\rm N}_{K/k}(U_{K,T}^+) : {\rm N}_{K/k}(U_{K,{\mathfrak m}}^+))$}

For a finite place ${\mathfrak P}$ of $K$, let $K_{\mathfrak P}$ be the ${\mathfrak P}$-completion of $K$
at ${\mathfrak P}$. Then let ${\mathcal U}_{K,\mathfrak P}$ be the group of local units of $K_{\mathfrak P}$ 
and ${\mathcal U}_{K,T} := \prod_{{\mathfrak P} \in T}{\mathcal U}_{K, \mathfrak P} \subset 
\prod_{{\mathfrak P} \in T}K_{\mathfrak P}^\times$; we denote by ${\mathcal U}_{K,{\mathfrak m}}$ 
the closure of $U_{K,{\mathfrak m}}^+$ in ${\mathcal U}_{K,T}$ ($T$ and ${\mathfrak m}$ seen in $K$).
We have analogous notations for the field $k$.

\smallskip
The arithmetical norm ${\rm N}_{K/k}=:{\rm N}$ can be extended by continuity 
on $\prod_{{\mathfrak P} \in T}K_{\mathfrak P}^\times$ and
the groups ${\rm N}({\mathcal U}_{K,T})$ and ${\rm N}({\mathcal U}_{K,{\mathfrak m}})$
are open compact subgroups of ${\mathcal U}_{k,T}$. It follows that the map 
${\rm N}(U_{K,T}^+) \ds\mathop{\too}^\theta {\rm N}({\mathcal U}_{K,T})/ {\rm N}({\mathcal U}_{K,{\mathfrak m}})$
is surjective. 

\smallskip
Consider its kernel. Let ${\rm N}(u)$, $u \in U_{K,T}^+$, be such that ${\rm N}(u) = 
{\rm N}(\alpha_{\mathfrak m})$, $\alpha_{\mathfrak m} \in {\mathcal U}_{K,{\mathfrak m}}$. 
Since  ${\rm H}^1(G, \prod_{{\mathfrak P} \in T}K_{\mathfrak P}^\times)=0$ 
(Shapiro's Lemma and Hilbert Theorem $90$), 
there exists $\beta \in \prod_{{\mathfrak P} \in T} K_{\mathfrak P}^\times$
such that $u = \alpha_{\mathfrak m} \beta^{1-\sigma}$.

\smallskip
We can approximate (over $T$) $\beta$ by $v \in K^{\times +}$ and $\alpha_{\mathfrak m}$ by 
$u_{\mathfrak m} \in U_{K,{\mathfrak m}}^+$; then $u = u_{\mathfrak m}\, v^{1-\sigma} \cdot \xi$,
with $\xi$ near from $1$ in $\prod_{{\mathfrak P} \in T}K_{\mathfrak P}^\times$ and totally positive; 
then let $u' = u\, v^{-(1-\sigma)}$; this leads to $u' = u_{\mathfrak m}\,\xi \in U_{K,{\mathfrak m}}^+$ and 
${\rm N}(u') ={\rm N}(u) \in {\rm N}(U_{K,{\mathfrak m}}^+)$.
The kernel of the map $\theta$ is ${\rm N}(U_{K,{\mathfrak m}}^+)$. Thus
\begin{equation}\label{eq22}
\begin{aligned}
({\rm N}(U_{K,T}^+) : {\rm N}(U_{K,{\mathfrak m}}^+)) 
&= ( {\rm N}({\mathcal U}_{K,T}): {\rm N}({\mathcal U}_{K,{\mathfrak m}}))
=\frac{({\mathcal U}_{k,T} : {\rm N}({\mathcal U}_{K,{\mathfrak m}}) )}
{({\mathcal U}_{k,T} : {\rm N}({\mathcal U}_{K,T}))} \\
&=\frac{({\mathcal U}_{k,T}  : {\mathcal U}_{k,{\mathfrak m}} )
\cdot ({\mathcal U}_{k,{\mathfrak m}} : {\rm N}({\mathcal U}_{K,{\mathfrak m}}))}
{({\mathcal U}_{k,T} : {\rm N}({\mathcal U}_{K,T}))}.
\end{aligned}
\end{equation}

By local class field theory we know that 
$({\mathcal U}_{k,T} : {\rm N}({\mathcal U}_{K,T})) = \prod_{{\mathfrak p} \in T} e_{\mathfrak p}$.
where $e_{\mathfrak p}$ is the ramification index of ${\mathfrak p}$ in $K/k$.

\begin{remark}\label{rema2}
 {\rm The index $({\mathcal U}_{k,{\mathfrak m}} : {\rm N}({\mathcal U}_{K,{\mathfrak m}}))$
may be computed from higher ramification groups in $K/k$ (cf. \cite[Chapitre V]{Se1}) by introduction
of the usual filtration of the groups ${\mathcal U}_{k,{\mathfrak p}}$ and ${\mathcal U}_{K,{\mathfrak P}}$.
If ${\mathfrak m} = \prod_{{\mathfrak p} \in T} {\mathfrak p}^{\lambda_{\mathfrak p}}$, $\lambda_{\mathfrak p}\geq 1$,
then ${\mathcal U}_{k,{\mathfrak m}} =  \prod_{{\mathfrak p} \in T}(1+
 {\mathfrak p}^{\lambda_{\mathfrak p}} {\mathcal O}_{\mathfrak p})$ and
 ${\mathcal U}_{K,{\mathfrak m}} =  \prod_{{\mathfrak p} \in T}\prod_{{\mathfrak P} \mid {\mathfrak p}}
 (1+ {\mathfrak P}^{\lambda_{\mathfrak p} e_{\mathfrak p}} {\mathcal O}_{\mathfrak P})$, where 
${\mathcal O}_{\mathfrak p}$ and ${\mathcal O}_{\mathfrak P}$ are the local rings of integers.
This local index only depends on the given extension $K/k$.}
\end{remark}

To go back to $\Cl_{k, {\mathfrak m}}^{+}$, we have  the following formula 
(cf. \cite[Corollary I.4.5.6 (i)]{Gr2})
\begin{equation}\label{eq23}
\# \Cl_{k, {\mathfrak m}}^{+} =  \# \Cl_k^{+} \cdot \frac{(U_{k, T}^+ : U_{k, {\mathfrak m}}^+)}
{(E_k^+ : E_{k, {\mathfrak m}}^+)} = \# \Cl_k^{+} \cdot 
\frac{({\mathcal U}_{k, T} : {\mathcal U}_{k, {\mathfrak m}})}
{(E_k^+ : E_{k, {\mathfrak m}}^+)} , 
\end{equation}

where the integer $({\mathcal U}_{k, T} : {\mathcal U}_{k, {\mathfrak m}})$ is given by
the generalized Euler function of ${\mathfrak m}$.

\smallskip
Then using \eqref{eq21}, \eqref{eq22}, \eqref{eq23}, we obtain the main result:

\begin{theorem} \label{thm1}
Let $K/k$ be a cyclic extension of Galois group $G$; let ${\mathfrak m}$ be a
nonzero integer ideal of $k$ and let $T$ be the support of ${\mathfrak m}$. Let $e_{\mathfrak p}$ 
be the ramification index in $K/k$ of any finite place ${\mathfrak p}$ of $k$. 
Then for any sub-$G$-module ${\mathcal H}$ of $\Cl_{K,{\mathfrak m}}^{+}$ 
and any subgroup ${\mathcal I}$ of $I_{K, T}$ such that 
${\mathcal I} \cdot P_{K,{\mathfrak m}}^+ / P_{K,{\mathfrak m}}^+ = {\mathcal H}$, we have
\begin{equation}\label{eq24}
\#\big( \Cl_{K ,{\mathfrak m}}^{+} /{\mathcal H}\big)^G= 
 \frac{\# \Cl_{k ,{\mathfrak m}}^{+}\cdot 
\prod_{{\mathfrak p} \notin T} e_{\mathfrak p} \cdot 
({\mathcal U}_{k ,{\mathfrak m}} :  {\rm N}({\mathcal U}_{K , {\mathfrak m}}))}
{[K : k] \cdot \#{\rm N}({\mathcal H})  
\cdot (\Lambda :  \Lambda \cap {\rm N}(U_{K,{\mathfrak m}}^+))} .
\end{equation}
where ${\rm N} = {\rm N}_{K/k}$ is the arithmetical norm and 
$\Lambda := \{x \in U_{k, {\mathfrak m}}^+,\ \, (x) \in {\rm N}({\mathcal I}) \}$.
\end{theorem}

Using, where appropriate, Lemma \ref{lemnorm}, we get the following corollaries:

\begin{corollary} {\rm \cite[Th\'eor\`eme 4.3, p. 41]{Gr3}.} Taking $T= \emptyset $, we obtain:
\begin{equation}\label{eq25}
\#\big( \Cl_K^{+} /{\mathcal H}\big)^G = 
 \frac{\# \Cl_k^{+}\cdot 
\prod_{{\mathfrak p} \in \Pl_{k,0}} e_{\mathfrak p}}{[K : k] \cdot \#{\rm N}({\mathcal H})  
\cdot (\Lambda  : \Lambda \cap {\rm N}(K^\times))} ,
\end{equation}

where $\Lambda := \{x \in k^{\times +}, \ \, (x) \in {\rm N}({\mathcal I}) \}$.
\end{corollary}

\begin{corollary} {\rm \cite[(1990)]{HL}.} If $T$ does not contain any prime ideal ramified in $K/k$, we obtain,
since in the unramified case $({\mathcal U}_{k ,{\mathfrak m}} :  {\rm N}({\mathcal U}_{K , {\mathfrak m}})) =1$ 
regardless~${\mathfrak m}$:
\begin{equation}\label{eq26}
\#\big( \Cl_{K,  {\mathfrak m}}^{+} /{\mathcal H}\big)^G = 
 \frac{\# \Cl_{k, {\mathfrak m}}^{+}\cdot 
\prod_{{\mathfrak p} \in \Pl_{k,0}} e_{\mathfrak p}}{[K : k] \cdot \#{\rm N}({\mathcal H})  
\cdot (\Lambda  : \Lambda \cap {\rm N}(K^\times))} .
\end{equation}
\end{corollary}

\begin{corollary} 
If $T= \emptyset $ and if ${\mathcal H} = \cl_K (S_K)$, where $S_K$ is any finite set of 
places of $K$, we obtain $\Lambda  /\Lambda \cap {\rm N}(K^\times) \simeq 
E_k^{{\rm N} S_K} / E_k^{{\rm N} S_K} \cap {\rm N}(K^\times)$ (see Remark \ref{IS}), and:
\begin{equation}\label{eq27}
\# \Cl_K^{S_K  G} = 
 \frac{\# \Cl_k^{+}\cdot 
\prod_{{\mathfrak p} \in \Pl_{k,0}} e_{\mathfrak p}}{[K : k] \cdot \# \cl_k (\langle {\rm N} S_K \rangle)
\cdot (E_k^{{\rm N} S_K}  : E_k^{{\rm N} S_K} \cap {\rm N}(K^\times))} ,
\end{equation}

where the group $E_k^{{\rm N} S_K}$ of ``${\rm N}S_K$-units'' is defined by:

\smallskip
$E_k^{{\rm N} S_K} = \{ x \in E_k^{S_k}, \, v_{\mathfrak p}(x) \equiv 0\! \pmod {f_{\mathfrak p}}
 \ \forall {\mathfrak p} \in S_k \  \& \   v_{\mathfrak p}(x) = 0  \
 \forall {\mathfrak p} \in \Pl_{k, \infty}\setminus S_{k, \infty}\}$, 

\smallskip
$S_k$ being the set of places of $k$ under $S_K$ and $f_{\mathfrak p}$ 
the residue degree of ${\mathfrak p}$.
\end{corollary}

\begin{proof} We have $\Lambda = \{x \in k^{\times +},\   (x) \in {\rm N}({\mathcal I}) \}$, where
${\mathcal I} = \langle {\mathfrak P} \rangle_{{\mathfrak P}\in S_{K,0}} \cdot 
\langle (y_{\mathfrak P}) \rangle_{{\mathfrak P}\in S_{K,\infty}}$. If $x \in \Lambda$, then
$(x)={\rm N}({\mathfrak A}) \cdot {\rm N}(A)$, 
${\mathfrak A} \in \langle {\mathfrak P} \rangle_{{\mathfrak P}\in S_{K,0}}$,
$A \in \langle (y_{\mathfrak P}) \rangle_{{\mathfrak P}\in S_{K,\infty}}$; hence,
up to ${\rm N} K^\times$, $x$ is represented by a ${\rm N} S_{K}$-unit $\varepsilon$.
One verifies that the map which associates $x$ with the image of $\varepsilon$ in 
$E_k^{{\rm N} S_K} / E_k^{{\rm N} S_K} \cap {\rm N}(K^\times)$ is well-defined and 
leads to the isomorphism. Note that $E_k^+ \subseteq E_k^{{\rm N} S_K}$.
\end{proof}

\begin{remark} \label{rema3} {\rm 
We have in \cite[p. 177 (1986)]{Ja1} another writing of this formula:
\begin{equation*}
\# \Cl_K^{S_K G} =  \frac{\# \Cl_k^{S_k}\cdot 
\prod_{{\mathfrak p} \notin S_k} e_{\mathfrak p} \cdot \prod_{{\mathfrak p} \in S_k} d_{\mathfrak p}}
{[K : k] \cdot (E_k^{S_k}  : E_k^{S_k} \cap {\rm N}(K^\times))},
\end{equation*}

where $d_{\mathfrak p} = e_{\mathfrak p}\,f_{\mathfrak p}$ is the local degree of $K/k$ at ${\mathfrak p}$
with $e_{\mathfrak p}=1$ for infinite places: use the relation $E_k^{S_k} \cap {\rm N}(K^\times) =
E_k^{{\rm N} S_K} \cap {\rm N}(K^\times)$ and the exact sequence

\medskip
\centerline{$1 \too E_k^{S_k} / E_k^{{\rm N} S_K} \tooo \langle S_k \rangle_\Z / \langle {\rm N} S_K \rangle_\Z
\tooo  \cl_k(\langle S_k \rangle_\Z )/ \cl_k( \langle {\rm N} S_K \rangle_\Z) \too 1$}

\medskip
for the comparison. Taking $S_K=\Pl_{K, \infty}$ in the two formulas, we get
\begin{equation}\label{eq28} 
\# \Cl_K^{\rm ord\, G} =  \frac{\# \Cl_k^{\rm ord}\cdot 
\prod_{{\mathfrak p} \in \Pl_{k,0}} e_{\mathfrak p} \cdot \prod_{{\mathfrak p} \in \Pl_{k,\infty}} f_{\mathfrak p}}
{[K : k] \cdot (E_k^{\rm ord}  : E_k^{\rm ord}  \cap {\rm N}(K^\times))}, 
\end{equation}

which is the true original Chevalley's formula (in the ordinary sense), where $f_{\mathfrak p}=2$
(resp. 1) if ${\mathfrak p} \in \Pl_{k,\infty}$ is complexified (resp. is not).}
\end{remark}

\subsection{Genera theory and heuristic aspects}\label{genera}
The usual case ($S = T = \emptyset$), in the cyclic extension $K/k$,
can be interpreted by means of the following diagram of finite extensions:
\unitlength=0.64cm
$$\vbox{\hbox{\hspace{-2.5cm}  \begin{picture}(11.5,6.1)
%     horizontales
\put(8.5,4.50){\line(1,0){2.0}}
\put(4.1,4.50){\line(1,0){2.5}}
\put(4.5,2.50){\line(1,0){2.2}}
\bezier{400}(4.0,4.8)(7.5,5.4)(10.6,4.8)
\put(7.2,5.4){$\Cl^{+}_K$}
\bezier{200}(4.5,2.3)(5.6,1.9)(6.7,2.3)
\put(4.8,1.6){${\rm N}\Cl^{+}_K$}
\bezier{300}(3.9,0.5)(6.7,0.5)(7.3,2.1)
\put(6.5,0.7){$\Cl^{+}_k$}
%     vertical
\put(3.50,2.9){\line(0,1){1.20}}
\put(3.50,0.9){\line(0,1){1.20}}
\put(7.50,2.9){\line(0,1){1.20}}
\put(10.8,4.4){$H^{+}_K$}
\put(6.9,4.4){$K H^{+}_k$}
\put(3.3,4.4){$K$}
\put(6.9,2.4){$H^{+}_k$}
\put(2.6,2.4){$K\! \cap\! H^{+}_k$}
\put(3.3,0.40){$k$}
\end{picture}   }} $$
\unitlength=1.0cm

Here $K\cap H^{+}_k/k$ is the maximal subextension of
$K/k$, unramified at finite places, and the norm map 
${\rm N}_{K/k} : \Cl_K^{+} \too \Cl_k^{+}$
is surjective if and only if $K \cap H_k^{+} = k$. 
So formula \eqref{eq25}
can be interpreted as follows (which will be very important for numerical computations);
using the relations 
$$\hbox{$[K : k] = [K : K\! \cap\! H^{+}_k]  \cdot [K \cap H_k^{+} : k]\ \ \ \& \ \ \ 
\# \Cl^{+}_k = \#  {\rm N} (\Cl^{+}_K) \cdot [K \cap H_k^{+} : k]$, }$$

we shall get a product of two integers
\begin{equation}\label{eq29}
\#\big( \Cl_K^{+} /{\mathcal H}\big)^G = 
\frac{\#  {\rm N} (\Cl^{+}_K)}{\#  {\rm N}({\mathcal H})} \cdot
\frac{\prod_{{\mathfrak p} \in \Pl_{k,0}} e_{\mathfrak p}}{[K : K\! \cap\! H^{+}_k]
\cdot (\Lambda  : \Lambda \cap {\rm N}(K^\times))} .
\end{equation}

Thus in the computations using a filtration $M_i$ (see Section \ref{sect4}), the 
$G$-modules ${\mathcal H} = \cl_K ({\mathcal I})$ are denoted 
$M_i = \cl_K ({\mathcal I}_i)$; the $M_i$ and ${\rm N}(M_i)$ 
will be increasing subgroups of $\Cl_K^{+}$ and $\Cl_k^{+}$, respectively,
so that $M_n=  \Cl_K^{+}$ for some $n$. 

\smallskip
Then we know that $\Lambda_i = \{x_i \in k^{\times +},\ \, (x_i) \in {\rm N}({\mathcal I}_i) \}$, 
which means that $x_i$, being the norm of an ideal and totally positive, is a local norm at each 
unramified finite place and at each infinite place (from Remark \ref{nrs}, ($\alpha$), ($\beta$)); so it 
remains to consider {\it the local norms at ramified prime ideals} since by the Hasse norm theorem, 
$x\in {\rm N}(K^\times)$ if and only if $x$ is a local norm everywhere (apart from {\it one} place).
This can be done by means of norm residue symbols computations of Remark \ref{nrs}, ($\gamma$),
in the context of ``genera theory'' (see the abundant literature on the subject, for instance from 
the bibliographies of \cite{Fr}, \cite{Fu}, \cite{Gr2}, \cite{L4}), so that the integers:

\medskip
\centerline{$\ds \frac{\prod_{{\mathfrak p} \in \Pl_{k,0}} e_{\mathfrak p}}{[K : K \cap H^{+}_k]
\cdot (\Lambda_i  : \Lambda_i \cap {\rm N}(K^\times))}, \ \, i \geq 0$, }

\medskip
are decreasing because of the injective maps 

\medskip
\centerline{$E_k^+/E_k^+ \cap {\rm N}(K^\times) \hookrightarrow \cdots \hookrightarrow 
\Lambda_i/\Lambda_i \cap {\rm N}(K^\times) \hookrightarrow
\Lambda_{i+1}/\Lambda_{i+1} \cap {\rm N}(K^\times)  \hookrightarrow \cdots$}

\medskip
giving increasing indices $(\Lambda_i  : \Lambda_i \cap {\rm N}(K^\times))$.

\medskip
Let $I_{\mathfrak p}(K/k)$ be the inertia groups (of orders $e_{\mathfrak p}$) 
of the prime ideals ${\mathfrak p}$ and put
\begin{equation}\label{eq30}
\Omega(K/k) = \Big \{ (\tau_{\mathfrak p})_{\mathfrak p} \in
\plus_{{\mathfrak p} \in \Pl_0} I_{\mathfrak p}(K/k),\ \  \prd_{{\mathfrak p} \in \Pl_0} \tau_{\mathfrak p} = 1\Big\} ; 
\end{equation}

we have the genera exact sequence of class field theory 
(interpreting the product formula of Hasse symbols, \cite[Proposition IV.4.5]{Gr2})
$$1 \too E_k^+/E_k^+ \cap {\rm N}(K^\times) \mathop{\tooo}^{\omega}
\plus_{{\mathfrak p} \in \Pl_0} I_{\mathfrak p}(K/k) 
\mathop{\tooo}^{\pi} {\rm Gal}(H_{K/k}^{+}/ H_k^{+}) \too 1, $$

where $H_{K/k}^{+} := H_K^{+\,\rm ab}$ is the genera field defined as the maximal 
subextension of $H_K^{+}$, Abelian over $k$, where $\omega $ associates 
with $x \in E_k^+$ the family of Hasse symbols 
$\Big(\Big(  \frac{x,\,K/k}{{\mathfrak p}}\Big)\Big)_{{\mathfrak p} \in \Pl_0}$
in $\plus_{{\mathfrak p} \in \Pl_0} I_{\mathfrak p}(K/k)$ (hence in $\Omega(K/k)$),
and where $\pi$ associates with $(\tau_{\mathfrak p})_{\mathfrak p} \in 
\plus_{{\mathfrak p} \in \Pl_0} I_{\mathfrak p}(K/k)$ the product  
$\prod_{\mathfrak p} \tau'_{\mathfrak p}$ of the lifts
$\tau'_{\mathfrak p}$ of the $\tau_{\mathfrak p}$, in the inertia groups of
$H_{K/k}^{+}/ H_k^{+}$ (these inertia groups generate 
the group ${\rm Gal}(H_{K/k}^{+}/ H_k^{+})$ which is the image of $\pi$); 
from the product formula, if $(\tau_{\mathfrak p})_{\mathfrak p}$ is in the image of $\omega$, 
then this product $\prod_{\mathfrak p} \tau'_{\mathfrak p}$
fixes both $H_k^{+}$ and $K$, hence $K H_k^{+}$. Thus $\pi (\Omega(K/k))=
{\rm Gal}(H_{K/k}^{+}/ K H_k^{+})$ with $\pi \circ \omega(E_k^+) = 1$, giving
the isomorphisms 
$$\Omega(K/k) \big / \omega(E_k^+) \simeq {\rm Gal}(H_{K/k}^{+}/ K H_k^{+})
\ \ \& \ \  \omega(E_k^+) \simeq E_k^+ \big / E_k^+ \cap {\rm N}(K^\times) . $$

We have $\# \Omega(K/k) =\ds \frac{\prod_{{\mathfrak p} \in \Pl_{k,0}} e_{\mathfrak p}}
{[K : K \cap H_k^{+}]}$ and $H_{K/k}^{+}$ being fixed by $(\Cl_K^{+})^{1-\sigma}$, we get

\centerline{$[H_{K/k}^{+} : K] = [H_k^{+} : K \cap H_k^{+}] \cdot
\ds \frac{\prod_{{\mathfrak p} \in \Pl_{k,0}} e_{\mathfrak p}}{[K : K \cap H^{+}_k]
\cdot (E_k^+  : E_k^+ \cap {\rm N}(K^\times))} = \# \Cl_K^{+ \,G}$ }

as expected.

\smallskip
Since $\Lambda_i$ contains $E_k^+$, we have $\pi \circ \omega(\Lambda_i  / E_k^+)
\subseteq {\rm Gal}(H_{K/k}^{+}/ K H_k^{+})$. Therefore we have at the final step $i=n$, 
using \eqref{eq29} for ${\mathcal H} =M_n = \Cl_K^{+}$,
$$(\Lambda_n  : \Lambda_n \cap {\rm N}(K^\times))  = 
\frac{\prod_{{\mathfrak p} \in \Pl_{k,0}} e_{\mathfrak p}}{[K : K \cap H_k^{+}]}
= \# \Omega(K/k),$$ 

whence $\omega_n(\Lambda_n) = \Omega(K/k)$ and 
$\pi_n \circ \omega_n (\Lambda_n  / E_k^+) = {\rm Gal}(H_{K/k}^{+}/ K H_k^{+})$,
which explains that an obvious heuristic is that $\# \Cl_K^{+}$ has no
theoretical limitation about the integer $n$
(but its structure may have some constraints, see Section \ref{sect4}).

\medskip
An interesting case leading to significant simplifications is when there is a single ramified
place ${\mathfrak p}_0$ in $K/k$; indeed, the product formula (from $\Omega(K/k) =1$) implies 
$(\Lambda  : \Lambda \cap {\rm N}(K^\times)) = 1$ and 
$\ds\frac{e_{{\mathfrak p}_0}}{[K : K \cap H_k^{+}]} =1$, so that formula \eqref{eq29} reduces to
$\ds\#\big( \Cl_K^{+} /{\mathcal H}\big)^G = \frac{\# {\rm N}(\Cl_K^{+})}{\# {\rm N}({\mathcal H})}$,
where $\# {\rm N}(\Cl^{+}_K) = [H_k^{+} : K \cap H_k^{+}]$ is known.
If ${\mathfrak p}_0$ is totally ramified, then 
$\ds\#\big( \Cl_K^{+} /{\mathcal H}\big)^G = \frac{\# \Cl_k^{+}}{\# {\rm N}({\mathcal H})}$.

From the above formulas (e.g., Formula \eqref{eq27}), we get some practical applications:

\begin{theorem} \label{genclass} Let $K/k$ be a cyclic $p$-extension of Galois group $G$.
Let $S_K$ be a finite set of non-complex places of $K$ such that $\cl_K (\langle S_K \rangle)$
is a sub-$G$-module.

Consider the $p$-class group $\Cl_K^{S_K}$, for which we have the formula

\smallskip
\centerline{$\ds  \# \Cl_K^{S_K\, G} = 
\frac{ \# {\rm N}(\Cl_K^{+})}{\# \cl_k ( \langle {\rm N}S_K \rangle)}
\cdot \frac{\prod_{{\mathfrak p} \in \Pl_{k,0}} e_{\mathfrak p}}{[K : K \cap H_k^{+}] 
\cdot (E_k^{{\rm N} S_K}  : E_k^{{\rm N} S_K} \cap {\rm N}(K^\times))}$.}

\smallskip\noindent
Then we have $\langle \cl_K(S_K) \rangle_\Z = \Cl_K^{+}$ (i.e.,  $S_K$ generates 
the $p$-class group of $K$) if and only if the two following conditions are satisfied:

\medskip
(i) ${\rm N}(\Cl_K^{+}) = \cl_k(\langle {\rm N} S_K \rangle)$,

\smallskip
(ii) $\ds (E_k^{{\rm N} S_K}  : E_k^{{\rm N} S_K} \cap {\rm N}(K^\times) )= 
\frac{\prod_{{\mathfrak p} \in \Pl_{k,0}} e_{\mathfrak p}}{[K : K \cap H_k^{+}]}
= \# \Omega(K/k)$ (see \eqref{eq30}).

\smallskip
If $K \cap H_k^{+}=k$ and if all places ${\mathfrak P} \in S_K$ are unramified 
of residue degree $1$ in $K/k$, the two conditions become:

\medskip
(i$'$) $\Cl_k^{+} = \cl_k(\langle S_k \rangle)$, where $S_k$ is the set of places ${\mathfrak p}$ 
under ${\mathfrak P} \in S_K$,

\smallskip
(ii$'$) $\ds (E_k^{S_k}  : E_k^{S_k} \cap {\rm N}(K^\times)) = 
\frac{\prod_{{\mathfrak p} \in \Pl_{k,0}} e_{\mathfrak p}}{[K : k]} = \# \Omega(K/k)$.
\end{theorem}

So, if the $p$-class group $\Cl_k^{+}$ is numerically known, to characterize a set $S_K$ 
of generators for $\Cl_K^{+}$ needs only local normic computations with the group
$E_k^{S_k}$ of $S_k$-units of $k$ which are known. Moreover, we can restrict ourselves 
to the case of $p$-class groups in a cyclic extension of degree $p$.

\begin{example}\label{82}
 {\rm Consider $K=\Q(\sqrt{82})$, $k=\Q$ and $p=2$ (the fundamental unit is of norm
$-1$, hence ordinary and narrow senses coincide). We shall use the primes $3$ and $23$ which
split in $K$, and prime ideals ${\mathfrak P}_3$ and ${\mathfrak P}_{23}$ above.
It is clear that the $2$-rank of the class group of $K$ is $1$ (usual Chevalley's formula \eqref{eq28}).
The conditions of the theorem are equivalent to $(E_\Q^{S_\Q}  : E_\Q^{S_\Q} \cap {\rm N}(K^\times)) = 2$
since the product of ramification indices is equal to $4$;  for instance, $E_\Q^{S_\Q} = \langle  3 \rangle$
for $S_K =\{ {\mathfrak P}_3\}$. 

\smallskip
We have to compute, for some $x \in \Q^{\times +}$
(norm of an ideal, thus local norm at each unramified place), the Hasse symbol $\Big(\ds\frac{x, K/\Q}{41}\Big)$ 
which is equal to $1$ if and only if $x$ is local norm at $41$ (which is equivalent to be global norm in 
$K/\Q$ because of the product formula $\Big(\ds\frac{x, K/\Q}{41}\Big) \cdot \Big(\ds\frac{x, K/\Q}{2}\Big) = 1$ 
and the Hasse norm theorem). 

\smallskip
But from the method recalled in Remark \ref{nrs}, we have to find an ``associate number''
$x'$ such that $x' \equiv 1 \pmod 8\ \&\ x' \equiv x \pmod {41}$, then to compute the Kronecker 
symbol $\Big(\ds\frac{82}{x'}\Big)$ (we have used the fact that the conductor of $K$ is $8\cdot 41$).

\smallskip
We compute that $x=3$ is not norm of an element of $K^\times$, whence 
${\mathfrak P}_3$ generates the $2$-class group of $K$ (for $x=3$, $x'=249$, 
and $\big(\frac{82}{249}\big) = -1$).
We can verify that ${\mathfrak P}_3$ is of order $4$ since the equation $u^2 -82 \cdot v^2 = 3^e$
(with gcd$(u, v)=1$) has no solution with $e=1$ or $e=2$, but ${\rm N}(73+8\,\sqrt{82})=3^4$;
however, the knowledge of $\# \Cl_K$ is not required to generate the class group.

\smallskip
Now we consider $x=23$ for which $x' =105$ and $\big(\frac{82}{105}\big) = 1$.
We compute that indeed $\frac{65 + 7\,\sqrt{82}}{3}$ is of norm $23$; this is given by the PARI instruction 
(cf. \cite{P}):
$$\hbox{\it bnfisnorm(bnfinit($x^2-82), 23$))}. $$

Then we can verify that
$23$ is not the norm of an integer; so we deduce that the class of ${\mathfrak P}_{23}$ does 
not generate the $2$-class group of $K$ and is of order $2$ (indeed, ${\rm N}(761+84\,\sqrt{82})=23^2$
giving ${\mathfrak P}_{23}^2 = (761+84\,\sqrt{82})$).}
\end{example} 

\begin{remark} \label{rema4}{\rm 
Another important fact is the relation $\nu_{K/k} = j_{K/k} \circ {\rm N}_{K/k}$ when some
classes of $k$ capitulate in $K$ (i.e., $j_{K/k}$ non-injective). It is obvious that the classes of
order prime to the degree $d$ of $K/k$ never capitulate; this explains that we shall
restrict ourselves to $p$-class groups in $p$-extensions.

\smallskip
The generalizations of Chevalley's formula do not take into account this phenomena since 
they consider only groups of the form ${\rm N}_{K/k}({\mathcal H})$ without mystery 
(when $\Cl_k^{+}$ is well known), contrary to ${\mathcal H}^{\nu_{K/k}}$.

\smallskip
This property of ${\rm N}_{K/k}$ is valid if $K/k$ is any Galois extension; if $K/k$ has no
unramified Abelian subextension $L/K$ (what is immediately noticeable !) then ${\rm N}_{K/k}$ is 
surjective, but possibly not $\nu_{K/k}$. We have given in \cite{Gr5}, \cite{Gr5$'$}, 
numerical setting of this to disprove some statements concerning
the propagation of $p$-ranks of $p$-class groups in $p$-ramified $p$-extensions $K/k$.}
\end{remark}

These local normic calculations deduced from Theorem \ref{thm1} have been extensively 
studied in concrete cases from the pioneer work of Inaba \cite[(1940)]{I},
in quadratic, cubic extensions, etc. and applied to non-cyclic extensions (dihedral ones, etc): see, e.g., 
\cite{Fr}, \cite{Re}, \cite{Gr3}, \cite{Gr3$'$}, \cite{Gr4}, \cite{HL}, \cite{L1} (in the semi-simple case of 
$G$-modules), \cite{Bol}, \cite{L2}, \cite{L3}, \cite{L4}, \cite{Kol}, \cite{KMS}, \cite{Ge1}, \cite{Ge2}, 
\cite{Ge3}, \cite{Kl}, \cite{Gr5}, \cite{Y1}, \cite{Y2}, and the corresponding references of all these papers ! 

\smallskip
These techniques may give information on some class field towers problems, capitulation problems, 
often with the use of quadratic fields (\cite{ATZ1}, \cite{ATZ2}, \cite{Go}, \cite{GW1}, \cite{GW2}, \cite{GW3}, 
\cite{Su}, \cite{SW}, \cite{Ter}, \cite{Mai}, \cite{Ma1}, \cite{Ma2}, \cite{Miy1}, \cite{Miy2}, \cite{MoMo}, 
some examples in \cite{Gr5} and numerical computations in \cite{Gr5$'$}, \cite{Gr13}, \cite{Ku} 
for capitulation in Abelian extensions, then many results of N. Boston, F. Hajir and Ch. Maire, 
and many others as these matters are too broad to be exposed here).

\section{Structure of $p$-class groups in $p$-extensions} \label{sect4}

\subsection{Recalls about the filtration of a $\Z_p[G]$-module $M$, with $G \simeq \Z/p\,\Z$}\label{ssect41}
Let $K/k$ be a cyclic extension of prime degree $p$, of Galois group $G = \langle \sigma \rangle$.

\smallskip
Let $\Cl_K^{+}$, $\Cl_k^{+}$ be the class groups in the narrow sense (same theory with 
the ordinary sense for any data). We shall look at the $p$-class groups $\Cl_K^{+} \otimes \Z_p$, 
$\Cl_k^{+} \otimes \Z_p$, still denoted $\Cl_K^{+}$, $\Cl_k^{+}$ thereafter, by abuse of notation. 

\smallskip
We consider the $\Z_p[G]$-module 
$M := \Cl_K^{+}$ for which we define the filtration evocated in Section \ref{sect1}:
$$M_{i+1}/M_i := (M/M_i)^G, \ \ M_0=1;$$
we denote by $n$ the least integer $i$ such that $M_i = M$.
For all $i \geq 0$ we have 
$$\hbox{$M_{i+1}^{1-\sigma} \subseteq M_i$, $\ M_i = \big\{h \in M,\   h^{(1-\sigma)^i}=1 \big\}$, 
and $\ \# M=\prod_{i=0}^{n-1} \#  (M_{i+1}/M_i)$. }$$

For all $i \geq 1$, the maps $\ds M_{i+1}/M_i \mathop{\tooo}^{1-\sigma} M_i/M_{i-1}$ 
are injective, giving a decreasing sequence for the orders $\#  (M_{i+1}/M_i)$ as $i$ grows,
whence $\# (M_{i+1}/M_i) \leq \# M_1$.

If for instance $\# M_1 = p$, then $\#  (M_{i+1}/M_i) = p$ for $0 \leq i \leq n-1$.

\smallskip
Remark that $\Cl_k^{+}$ has no obvious $G$-module definition from $M$
(it is not isomorphic to
$M_1=M^G$, nor to $M^{\nu_{K/k}}$ for $\nu_{K/k} := 1+ \sigma+ \cdots + \sigma^{p-1}$); this is 
explained by the difference of nature between $\nu_{K/k}$ and the arithmetical norm ${\rm N}_{K/k}$
of class field theory.

\subsection{Case $M^\nu=1$} \label{ssect42}
When $M^\nu=1$ for $\nu := \nu_{K/k} = 1+ \sigma+ \cdots + \sigma^{p-1}$, 
$M$ is a $\Z_p[G]/(\nu)$-module and we have
$$\Z_p[G]/(\nu) \simeq \Z_p[X]/(1+X+\cdots +X^{p-1}) \simeq \Z_p[\zeta], $$

where $\zeta$ is a primitive $p$th root of unity; then we know that 
$$M \simeq \plus_{j=1}^m \Z_p[\zeta]/ (1-\zeta)^{n_j},\ \, 
1 \leq n_1 \leq n_2 \leq \cdots \leq n_m, \ \, m \geq 0 ,$$

whose $p$-rank can be arbitrary. The exact sequence
$$1 \too M_1 = M^G  \tooo M \displaystyle \mathop{\tooo}^{1-\sigma }  M^{1-\sigma } \too 1$$

becomes in the $\Z_p[\zeta]$-structure:
\begin{equation}\label{eq31}
\begin{aligned}
1 \too \plus_{j=1}^m &(1-\zeta )^{n_j - 1}\, \Z_p[\zeta] / (1-\zeta)^{n_j}  \tooo \\
&M = \plus_{j=1}^m \Z_p[\zeta]/ (1-\zeta)^{n_j} 
\displaystyle \mathop{\tooo}^{1-\zeta}  \plus_{j=1}^m (1-\zeta)\, \Z_p[\zeta]/ (1-\zeta)^{n_j} \too 1,
\end{aligned}
\end{equation}

where the submodules $M_i$ are given by $M_i = \plus_{j,\,  n_j \leq i} \Z_p[\zeta]/ (1-\zeta)^{n_j}$
(for $0\leq i \leq n$, where $n = n_m$).

Each factor $N_j :=  \Z_p[\zeta]/ (1-\zeta)^{n_j}$ (such that $M=\plus_{j=1}^m N_j$, not to be
confused with the $M_i = \plus_{j,\,  n_j \leq i} N_j$) has a structure of group given by the following result:

\begin{theorem}\label{structure}
Under the assumption $M^{\nu_{K/k}} = 1$ in the cyclic extension $K/k$ of degree $p$,
put $n_j=a_j\,(p-1) + b_j$, $a_j \geq 0$ and $0 \leq b_j \leq p-2$, in the decomposition of $M$ 
in elementary components as above. Then 
\begin{equation}\label{eq32}
\begin{aligned}
&N_j := \Z_p[\zeta]/ (1-\zeta)^{n_j} \simeq (\Z/p^{a_j+1}\Z)^{b_j} \plus (\Z/p^{a_j}\Z)^{p-1-b_j} ,\ \ \,
\forall j=1, \ldots, m. \\
&M \simeq \plus_{j=1}^m \Big[ (\Z/p^{a_j+1}\Z)^{b_j} \plus (\Z/p^{a_j}\Z)^{p-1-b_j}\Big ]. 
\end{aligned}
\end{equation}
\end{theorem}

\begin{proof} We have $N_j :=\Z_p[\zeta]/ (1-\zeta)^{n_j} \simeq \Z_p[\zeta]/ p^{a_j} \,(1-\zeta)^{b_j}$.
So, to have the structure of  group, it is sufficient to compute the $p^k$-ranks
for all $k \geq 1$ (i.e., the dimensions over $\F_p$ of $N_j^{p^{k -1}}/N_j^{p^k}$), which is
immediate since this $p^k$-rank is $p-1$ for $k \leq a_j$, $b_j$ for $k=a_j+1$, and 0 for $k > a_j+1$.
\end{proof}

This implies that the $p$-rank of $N_j$ is $p-1$ if $a_j\geq 1$  and $b_j$ if $a_j=0$ 
(i.e., $b_j=n_j \leq p-2$). 
So the parameters $a_j$ and $b_j$ will be important in a theoretical and numerical point of view. Put 
$M^{(k)} := \big\{h \in M,\   h^{p^k}=1 \big\}, \ \ k\geq 0$.

\begin{lemma}\label{lem41} If $M^\nu = 1$, then $M^{(k)} = M_{k\cdot(p-1)},  \forall k\geq 0$,
and the $p^k$-rank $R_k$ of~$M$ is the $\F_p$-dimension of $M^{(k-1)} / M^{(k)}$. 
Then $p^{R_k} = \prod_{i=(k-1)(p-1)}^{k(p-1)-1} \# (M_{i+1} / M_i)$.
\end{lemma}

\begin{proof} Immediate from the $\Z_p[\zeta]$-structure and properties of Abelian $p$-groups.
\end{proof}

\subsection{Case $M^\nu \ne1$} \label{ssect43}
We have, in the same framwork, the following result
in the case $M^\nu \ne 1$, but $\# (M_{i+1}/M_i) = p$ \cite[Proposition 4.3, pp. 31--32]{Gr2}:

\begin{theorem}\label{thm3}
 Let $K/k$ be a cyclic extension of prime degree $p$, of Galois group 
$G = \langle \sigma \rangle$ and let $M$ be a finite $\Z_p[G]$-module such that
$M^{\nu_{K/k}} \ne 1$. Let $n$ be the least integer $i$ such that $M_i=M$. We assume that
$\# M_1 = p$. 

\smallskip
Put $n = a\cdot (p-1)+ b$, with $a \geq 0$ and $0 \leq b \leq p-2$.
Then we have necessarily $n\geq 2$ and the following possibilities:

\smallskip
(i) Case $n < p$. Then $M \simeq \big( \Z/ p^{2} \Z \big) \plus \big( \Z/ p \Z \big)^{n-2}$.

\smallskip
(ii) Case $n=p$. Then $M \simeq \big( \Z/ p \Z \big)^p$ or $\big( \Z/ p^{2} \Z \big) \plus \big( \Z/ p \Z \big)^{p-2}$.

\smallskip
(iii) Case $n>p$. Then $M \simeq \big( \Z/ p^{a+1} \Z \big)^b \plus \big( \Z/ p^{a} \Z \big)^{p-1-b}$.
\end{theorem}

\begin{proof}
The proof needs two lemmas (in which we keep the notation $M_n$ for $M$).

\begin{lemma}\label{lemma1}
For all $k \geq 1$ we have the exact sequence
\begin{equation}\label{eq33}
1 \too M_1 \cap M_n^{p^{k-1}} / M_1 \cap M_n^{p^k} \tooo M_n^{p^{k-1}} / M_n^{p^k} 
\mathop{\tooo}^{1-\sigma} M_{n-1}^{p^{k-1}} / M_{n-1}^{p^k} \too 1. 
\end{equation}
\end{lemma}

\begin{proof}
Under the assumption $\# M_1 = p$, we know fron \S\,\ref{ssect41} that
$\# (M_{i+1} / M_i) = p$, $0 \leq i \leq n-1$; we have the exacte sequence
$1 \to M_1 \too M_{i+1} \ds \mathop{\too}^{1-\sigma} M_{i+1}^{1-\sigma} \to 1$,
which shows that $\# (M_{i+1} / M_{i+1}^{1-\sigma}) = p$, hence $M_{i+1}^{1-\sigma} = M_i$
since $M_{i+1}^{1-\sigma} \subseteq M_i$.

\medskip
Let $x \in M_n^{p^{k-1}}$ such that $x^{1-\sigma} = y^{p^k}$, $y \in M_{n-1}$. 
There exists $z \in M_n$ such that $y = z^{1-\sigma}$ and $x^{1-\sigma}=z^{p^k\cdot (1-\sigma)}$;
thus $(x\cdot z^{- p^k})^{1-\sigma} = 1$ so that $x\cdot z^{- p^k} \in M_1 \cap M_n^{p^{k-1}} $, 
giving
$$ {\rm Ker}(1-\sigma)  \subseteq M_1 \cap M_n^{p^{k-1}} / M_1 \cap M_n^{p^k}, $$
the opposite inclusion being obvious as well as the surjectivity.
\end{proof}

\begin{lemma}\label{lemma2}
If $n \ne p$ then the $p$-rank of $M_n$ is equal to the $p$-rank of $M_{n-1}$.
\end{lemma}

\begin{proof}
From the relation $(1 - \zeta)^{p-1} =p \cdot A(\zeta)$, where $A(\zeta) \equiv -1 \pmod {(1 - \zeta)}$,
we have $\nu = (1 - \sigma)^{p-1} - p \cdot A(\sigma)$, $A(\sigma) \equiv -1 \pmod {(1 - \sigma)}$
(i.e., $A(\sigma)$ invertible in $\Z_p[G]$).

\medskip
(a) Case $n > p$. Let $x \in M_{n-1}\setminus M_{n-2}$ (this makes sense since $n \geq p+1 \geq 3$)
and let $y = x^{(1-\sigma)^{n-2}}$; then $y \in M_1$, $y \ne 1$ because of the choice of $x$. There
exists $B(\sigma) \in \Z_p[G]$ such that $(1-\sigma)^{n-2} = B(\sigma)\cdot (1-\sigma)^{p-1}$ and
with $z = x^{B(\sigma)}$ one obtains $y = z^{(1-\sigma)^{p-1}}$. Since $M_{n-1} = M_n^{1-\sigma}$
one gets $M_{n-1}^\nu = 1$, so that $z^\nu=1$ and $ z^{(1-\sigma)^{p-1}} = z^{p \,\cdot A(\sigma)}$
which shows that $y\in M_n^p$ ; the assumption $\# M_1 = p$ implies the inclusion
$M_1 \subseteq M_n^p$ (in fact $y\in M_{n-1}^p$). The exact sequence \eqref{eq33} 
applied with $k=1$ leads to the isomorphism $M_n / M_n^p \simeq M_{n-1} / M_{n-1}^p$.

\medskip
(b) Case $n < p$. So $M_{n-1} \simeq (\Z/p\,\Z)^{n-1}$ (Theorem \ref{structure} 
applied to $M_{n-1} = M_n^{1-\sigma}$);
but the relation $\nu = (1-\sigma)^{p-1} - p\cdot A(\sigma)$ leads to
$M_n^\nu = M_n^{(1-\sigma)^{p-1} - p\cdot A(\sigma)} = M_n^p $
because $M_n^{(1-\sigma)^{p-1}} = 1$; since $M_n^\nu = M_1$, we get $M_n^p = M_1$
and necessarily 
$$M_n \simeq \big( \Z/ p^{2} \Z \big) \plus \big( \Z/ p \Z \big)^{n-2}.$$

These computations lead to the cases (i) and to a part of (ii) of the theorem since, in 
the case $n=p$, the exact sequence \eqref{eq33} for $k=1$ is 
$1 \to M_1 / M_1 \cap M_n^{p} \to M_n / M_n^{p} \to M_{n-1} / M_{n-1}^{p} \to 1$, 
and the structure depends on the order ($1$ or $p$) of the kernel contrary 
to the previous case.
\end{proof}

We have to prove the point (iii) of the theorem using (a) of the lemma.
 We then suppose $n>p$. We note that, with obvious notation, $(M_i)_j = M_j$
for $j \leq i$; so we can apply Theorem \ref{structure} to $M_{n-1}$. Lemma \ref{lemma1} shows that
the $p^k$-rank of $M_n$ is larger than (or equal to) that of $M_{n-1}$; as the $p^k$-rank of a group is a 
decreasing function of $k$,  Lemma~\ref{lemma2} and the above remark show that for 
$k \leq \big [\frac{n-1}{p-1} \big]$, the $p^k$-ranks of $M_n$ and $M_{n-1}$ are equal to $p-1$.

\smallskip
Put $n-1 = a' \, (p-1) + b'$, $0 \leq b' \leq p-2$ (in fact $a' =\big [\frac{n-1}{p-1}\big]$).

\smallskip
The exact sequence of Lemma \ref{lemma1} shows a priori three possibilities:

\medskip
($\alpha$) Case $b'=0$. Necessarily, $R_{a'+1}(M_n) = 1$ and $R_{a'+1}(M_{n-1}) = 0$.

\medskip
($\beta$) Case $b'>0$ and $R_{a'+1}(M_n) = R_{a'+1}(M_{n-1}) + 1$.

\medskip
($\gamma$) Case $b'>0$ and $R_{a'+1}(M_n) = R_{a'+1}(M_{n-1})$ and $R_{a'+2}(M_n) = 1$.

\medskip
So it remains to prove that the case ($\gamma$) is not possible. Let $x\in M_n$,
$x \notin M_{n-1}$; we have $x^\nu \in M_1\   \& \  x^\nu = x^{(1-\sigma)^{p-1}} \cdot x^{- p\cdot A(\sigma)}$;
put $x' := x^{(1-\sigma)^{p-1}}$ and $x'' = x^{- p\cdot A(\sigma)}$; we have
$x' \in M_{n-(p-1)} = M_{(a'-1)(p-1) + b' +1} \subset M_{a' (p-1)}$;
but $M_{a' (p-1)} = (M_{n-1})_{a' (p-1)} = (M_{n-1})^{(a')}$. As $x \notin M_{n-1}$, we have
$x^{p^{a'+1}} \ne 1$, hence $x''{}^{p^{a'}} \ne 1$. Thus we have obtained
$x' \in  (M_{n-1})^{(a')}$ and $x'' \notin  (M_{n-1})^{(a')}$; since $x^\nu \in M_1$ and $a' \ne 0$
(we have $n \geq p+1$), one has $x^\nu \in (M_{n-1})^{(a')}$, in other words
$x'' = x^\nu \cdot x'{}^{-1} \in (M_{n-1})^{(a')}$ (absurd).
\end{proof}

This finishes a particular case of structure when $M^{\nu_{K/k}}$ is not specified. Of course, we 
have $M^{\nu_{K/k}} \subseteq M_1$ and when $\# M_1 = p$, we have $\# M^{\nu_{K/k}} = 1$ or $p$. 
It would be interesting to have more general structure theorems.

\subsection{Numerical computations for $p$-class groups}
Now we apply these results to the $p$-class group $M=\Cl_K^{+}$ in $K/k$ 
cyclic of degree $p$. Many cases are possible: 

\smallskip
If the transfer map $j_{K/k}$ is injective
then $(\Cl_K^{+})^{\nu_{K/k}} \simeq {\rm N}_{K/k}(\Cl_K^{+})$.

\smallskip
The map ${\rm N}_{K/k}$ is surjective except if
$K/k$ is unramified (i.e., $K \subset H_k^{+}$, the $p$-Hilbert class field of $k$); 
if $K/k$ is ramified we get ${\rm N}_{K/k}(\Cl_K^{+})=\Cl_k^{+}$.

\smallskip
The transfer map may be non-injective while ${\rm N}_{K/k}$ is surjective, which causes 
more intricate theoretical calculations. But as we know, if ${\rm N}_{K/k}$ is not surjective
(unramified case), then $j_{K/k}$ is never injective (Hilbert's Theorem 94, 
\cite{GW1}, \cite{GW2}, \cite{GW3}, \cite{Su}, \cite{Ter}). 

\medskip
To simplify, we suppose $K/k$ cyclic of degree $p$ and not unramified
(otherwise, we get $\ds \# M_1 = \frac{\# \Cl_k^{+}}{p}$ and more generally
$\ds \#( M_{i+1} / M_i) =  \frac{\# \Cl_k^{+} }{p\, \cdot\, \#{\rm N}(M_i)}$,
which can be carried out in the same way).

\medskip
We suppose that $K/k$ is ramified at some prime ideals ${\mathfrak p}_1, \ldots , {\mathfrak p}_t$ of $k$ 
($t \geq 1$). We make no assumptions about $\# M_1$ and $M^{\nu_{K/k}}$.
With the previous notations and definitions, we then have the simplified formulas \eqref{eq25}
for which the submodule ${\mathcal H}$ is an element $M_i =: \cl_K({\mathcal I}_i)$ of the filtration of $M$:
$$\#\big( \Cl_K^{+} /M_i \big)^G = \# (M_{i+1} /M_i) = \frac{\# \Cl_k^{+}\cdot 
\prod_{{\mathfrak p} \in \Pl_{k,0}} e_{\mathfrak p}}{[K : k] \cdot \#{\rm N}(M_i)  
\cdot (\Lambda_i  : \Lambda_i \cap {\rm N}(K^\times))} , $$

where $\Lambda_i = \{x_i \in k^{\times +},\ \, (x_i) \in {\rm N}({\mathcal I}_i) \}$.
If $p>2$ one can use the ordinary sense and remove the mention ${}^{+}$ in all
the forthcoming expressions. 

\medskip
(i) Computation of $M_1=M^G$ from $M_0=1$, which means that ${\mathcal I}_0=1$, hence
${\rm N}(M_0)=1$ and $\Lambda_0=\{x_0 \in k^{\times +},\  (x_0) \in {\rm N} (1)\} = E_k^+$,
giving the following expression where we have put  $(E_k^+  : E_k^+ \cap {\rm N}(K^\times)) =: p^{\delta_0}$:
\begin{equation}\label{eq34}
\begin{aligned}
\#( M_1/ M_0) & = \#\Cl_K^{{+}\,G} =  \frac{\# \Cl_k^{+}\cdot  p^{t-1}}
{ (\Lambda_0  : \Lambda_0 \cap {\rm N}(K^\times))}  \\
& =: \# \Cl_k^{+}\cdot  p^{t-1-\delta_0}  .
\end{aligned}
\end{equation}

First we remark that we have the isomorphism:
$$\Cl_K^{+\, G} \big / \cl_K(I_K^G) \simeq E_k^+ \cap {\rm N}(K^\times) / {\rm N}(E_K^+),$$
which shows how to obtain $M_1=\Cl_K^{+\, G}$ from $\cl_K(I_K^G)$ (called the group of 
strongly ambiguous classes) and {\it global normic computations} 
with units of $k$. But the group ${\rm N}(E_K^+)$ is not effective and we must proceed otherwise.
In other words, the group of strongly ambiguous classes $\cl_K(I_K^G)$ is not a ``local'' invariant, 
contrary to $\Cl_K^{+\, G}$.

\smallskip
So in the first step (which is a bit particular since ${\mathcal I}_0=1$ and $\Lambda_0 = E_k^+$),
we shall look at the $x_0 \in \Lambda_0$ which are norms of some $y_1 \in K^{\times +}$. 

\smallskip
So $(x_0) =  {\rm N}(y_1) = (1)$, $y_1 \in K^{\times +}$, which yields 
$(y_1) \cdot {\mathfrak A}_1^{1-\sigma} = 1$, where ${\mathfrak A}_1$ is defined up to an invariant 
ideal, so that ${\mathcal I}_1$ contains at least such non-invariants ideals 
${\mathfrak A}_1^1, \ldots, {\mathfrak A}_1^{r_1}$, and invariant ideals 
(in which are ideals ${\mathfrak a}^1, \ldots, {\mathfrak a}^s$, generating $\Cl_k^{+}$, 
extended to $K$, and ramified prime ideals ${\mathfrak P}^1, \ldots,{\mathfrak P}^t$). 

\smallskip
Reciprocally, if $\cl_K({\mathfrak A}'_1) \in M_1$, there exists $y_1 \in K^{\times +}$ 
such that $(y_1) \cdot {\mathfrak A}'_1{}^{1-\sigma} = (1)$, giving ${\rm N}(y_1) = x_0 \in \Lambda_0$.

\smallskip
Thus, it is not difficult to see that the classes of these ideals generate $M_1$, whence
$${\mathcal I}_1=\{ {\mathfrak A}_1^1, \ldots, {\mathfrak A}_1^{r_1}\ ; \ 
 ({\mathfrak a}^1), \ldots, ({\mathfrak a}^s) \ ; \  {\mathfrak P}^1, \ldots,{\mathfrak P}^t \} .$$

This gives ${\rm N}(M_1)$ by means of the computation, in $\Cl_k^{+}$,
of ${\rm N}({\mathcal I}_1)$ ($M_1$ does not need to be computed as a subgroup 
of $ \Cl_K^{+}$), then, with $\Lambda_1 = \{x_1 \in k^{\times +}, \ (x_1) \in {\rm N}({\mathcal I}_1)\}$:
\begin{equation}\label{eq35}
\begin{aligned}
\#( M_2 /M_1) & =  \frac{\# \Cl_k^{+}\cdot p^{t-1}}{\# {\rm N}(M_1)  
\cdot (\Lambda_1  : \Lambda_1 \cap {\rm N}(K^\times))} \\
& =:  \frac{\# \Cl_k^{+}}{\# {\rm N}(M_1) } \cdot p^{t-1-\delta_1} .
\end{aligned}
\end{equation}

\begin{remark} The $p$-class group $\Cl_K^{+}$ is equal to the group of ambiguous classes if and only if 
$\cl_k( {\rm N} {\mathcal I}_1) = \Cl_k^{+}$ \& $\delta_1 = t-1$. If $\Cl_k^{+}=1$, the group
$\Lambda_1$ is easily obtained from ${\rm N} {\mathcal I}_1 \subset P_k^+$, whence the computation of 
$\delta_1$; since ${\mathcal I}_1$ only depends on $E_k^+ \cap {\rm N}K^\times$ and the ramification in $K/k$,
we can hope to characterize the fields $K$ fulfilling these conditions.
\end{remark}

(ii) For the computation of ${\mathcal I}_2$, we process from the elements of 
$\Lambda_1$ which are norms of some $y_2 \in K^{\times +}$
and the analogous fact that if $x_1 \in \Lambda_1$ is norm, then
$(x_1) = {\rm N}(y_2) = {\rm N}({\mathfrak B}_1)$, ${\mathfrak B}_1 \in {\mathcal I}_1$,
$y_2 \in K^{\times +}$, hence there exists ${\mathfrak A}_2 \in I_K$
such that ${\mathfrak B}_1= (y_2) \cdot {\mathfrak A}_2^{1-\sigma}$.

\smallskip
Reciprocally, let $h_2 =\cl_K({\mathfrak A}'_2) \in M_2$ for some ${\mathfrak A}'_2 \in I_K$; since
$h_2^{1-\sigma}\in M_1$, there exists $y_2 \in K^{\times +}$ such that 
$(y_2)\cdot {\mathfrak A}'_2{}^{1-\sigma} = {\mathfrak A}'_1 \in {\mathcal I}_1$, hence
${\rm N}({\mathfrak A}'_1) = {\rm N}(y_2) =: (x_1)$, $x_1 \in  \Lambda_1$
(since for all $i$, $E_k^+ \subseteq \Lambda_i$ and invariant ideals are in ${\mathcal I}_i$,
the choices of $x_2$ and ${\mathfrak A}'_2$ do not matter).

\smallskip
Then these ideals of the form ${\mathfrak A}_2^1, \ldots, {\mathfrak A}_2^{r_2}$ must be 
added to ${\mathcal I}_1$ to create ${\mathcal I}_2$:
$${\mathcal I}_2=\{ {\mathfrak A}_1^1, \ldots, {\mathfrak A}_1^{r_1}\ ; \ 
 {\mathfrak A}_2^1, \ldots, {\mathfrak A}_2^{r_2} \ ; \ 
 ({\mathfrak a}^1), \ldots, ({\mathfrak a}^s) \ ; \  {\mathfrak P}^1, \ldots,{\mathfrak P}^t  \} ,$$

whence ${\rm N}(M_2)$ and 
$$\Lambda_2 = \{x_2 \in k^{\times +}, \ (x_2) \in {\rm N}({\mathcal I}_2)\}, $$

and so on.
Hence, the algorithm is very systematic and the use of normic symbols to find the
subgroups $\Lambda_i \cap {\rm N}(K^\times)$ is effective: 
indeed, for the most general case of computation of Hasse symbols, see the Remark \ref{nrs} below;
otherwise use Hilbert symbols $(x_i, \alpha)_{\mathfrak p}$ by adjunction to $k$ of 
a primitive $p$th roots of unity $\zeta_p$ to obtain the Kummer extension 
$$K' := K(\zeta_p) =: k' (\sqrt[p]{\alpha\,}),\ \ \alpha \in k'{}^\times,$$

over $k' := k(\zeta_p)$, and use the obvious Galois structure in $K' / k'/ k$ for the radical $\alpha$ and 
the decomposition of ramified prime ideals, i.e., the duality of 
characters given by the reflection principle \cite[\S\S\,II.1.6.8,  II.5.4.2, II.5.4.3, II.7.1.5,  II.7.5]{Gr2}; 
this leads to generalizations of R\'edei's matrices over $\F_p$;
the rank of the matrices, denoted $\delta_i$, may be introduced in the general formula to give:
\begin{equation}\label{eq36}
\#( M_{i+1} /M_i) =  \frac{\# \Cl_k^{+}}{ \# {\rm N}(M_i)} \cdot p^{t-1 - \delta_i} ,
\end{equation}

with increasing $\delta_i$ up to the value $i=n$ giving $\delta_i = t-1$ and 
$\# {\rm N}(M_i) = \# \Cl_k^{+}$.

\smallskip
This was done in \cite[(1973)]{Gr3$'$} essentially for $p=2, 3$, and in \cite[Theorem 5.16 (2015)]{KMS}, 
for $p=5$, when the base field contains $\zeta_p$ and for particular $\alpha$ (essentially $k=\Q(\zeta_5)$
and $K=k(\sqrt[5]{q})$ where $q \in \N$ is for instance a prime satisfying some conditions, 
so that the $5$-rank can be bounded explicitly by a precise computation of the filtration);
this approach by \cite{KMS} applies to the arithmetic of elliptic curves in the $\Z_5$-extension of $k$.

\begin{remark} \label{nrs}
{\rm For convenience, recall (from \cite[II.4.4.3]{Gr2}) the hand computation of normic Hasse 
symbols $\big(\frac{x\,,\,{K}/k}{{\mathfrak p}} \big)$, by global means, in {\it any Abelian extension $K/k$}.

\smallskip
Let $\mathfrak m$ be a multiple of the conductor $\mathfrak f$ of $K/k$ (it does not matter if the support $T$
of $\mathfrak m$ strictly contains the set of (finite) places ramified in $K/k$, which will be the case if the 
conductor is not precisely known).
Set ${\mathfrak m} =: \prd_{{\mathfrak p} \in T} {\mathfrak p}^{m_{\mathfrak p}}$ with $m_{\mathfrak p} > 0$.

Let $x\in k^\times$ and let ${\mathfrak p}$ be a place of $k$
($x$ is not assumed to be prime to ${\mathfrak p}$); let us consider several 
cases, where $\big(\frac{K/k}{{\mathfrak p}} \big)$ denotes the Frobenius automorphism
of ${\mathfrak p}$ in $K/k$ (for an unramified ${\mathfrak p}$; for an infinite complexified place, the 
Frobenius is a complex conjugation), and let $v_{\mathfrak p}$ be the ${\mathfrak p}$-adic valuation:

\smallskip
\ ($\alpha$) ${\mathfrak p}\in \Pl_{\infty}$ (real infinite place). We have
$\big(\frac{x\,,\,{K}/k}{{\mathfrak p}} \big) = 
\big(\frac{K/k}{{\mathfrak p}} \big)^{{v_{\mathfrak p}}(x)}$, 
where ${{v_{\mathfrak p}}(x)} = 0$ (resp. 1) if $\sigma_{\mathfrak p}(x) >0$ (resp. $\sigma_{\mathfrak p}(x)<0$).

\smallskip
\ ($\beta$) ${\mathfrak p} \in \Pl_0 \setminus T$. Similarly, since ${\mathfrak p}$ is unramified, we have
$\big(\frac{x\,,\,{K}/k}{{\mathfrak p}} \big) = \big(\frac{K /k}{{\mathfrak p}} \big)^{{v_{\mathfrak p}}(x)}$.

\smallskip
\ ($\gamma$) ${\mathfrak p}\in T$. Let $x' \in k^\times$ (called a ${\mathfrak p}$-associate of $x$) 
be such that (using the multiplicative Chinese remainder theorem):

\smallskip
\quad (i) $x' x^{-1} \equiv 1 \pmod {{\mathfrak p}^{m_{\mathfrak p}}}$,

\smallskip
\quad (ii) $x' \equiv 1 \pmod {{\mathfrak p'}^{m_{\mathfrak p'}}}$), for each place 
${\mathfrak p}'\in T,\ {\mathfrak p}'\ne {\mathfrak p}$,

\smallskip
\quad (iii) $\sigma_{{\mathfrak p}'}(x') >0$ for each infinite place ${\mathfrak p}' \in \Pl_{\infty}$, 
complexified in $K /k$.

\smallskip
Then, by the product formula, we have
$\big(\frac{x'\,,\,{K}/k}{{\mathfrak p}} \big) = \prod_{{\mathfrak p}'\in \Pl,{\mathfrak p}'\ne {\mathfrak p}}\big
(\frac{x'\,,\,{K}/k}{{\mathfrak p}'} \big)^{-1}$, and since
$\big( \frac {x\,,\,{K}/k}{{\mathfrak p}} \big) = \big(\frac {x'\,,\,{K}/k}{{\mathfrak p}} \big)$
by (i) and the definition of the local ${\mathfrak p}$-conductor of $K/k$, we have
$\big(\frac{x\,,\,{K}/k}{{\mathfrak p}} \big) = \prod_{{\mathfrak p}'\in \Pl,{\mathfrak p}'\ne {\mathfrak p}}
\big (\frac{x'\,,\,{K}/k}{{\mathfrak p}'} \big)^{-1}$;
let us compute the symbols occurring in the right hand side:

\smallskip
$\quad \quad\bullet\ $ if ${\mathfrak p}' \in T \setminus \{{\mathfrak p}\}$, 
$x' \equiv 1 \pmod {{\mathfrak p'}^{m_{\mathfrak p'}}}$ (by (ii)) and we have
$\big(\frac {x'\,,\,{K}/k}{{\mathfrak p}'} \big) = 1$,

\smallskip
$\quad \quad\bullet\,$ if ${\mathfrak p}' \in \Pl_\infty$, $\big( \frac{x'\,,\,{K}/k}{{\mathfrak p}'} \big) = 1$
since either $\big( \frac {K /k}{{\mathfrak p}'} \big) = 1$ if ${\mathfrak p}'$ is complex or
non-complexified real, or ${v_{\mathfrak p}}'(x') = 0$ for ${\mathfrak p}'$ complexified real
(by (iii)),

\smallskip
$\quad \quad\bullet\ $ if ${\mathfrak p}' \in \Pl_0 \setminus T$, ${\mathfrak p}'$ is unramified 
and we know that $\big (\frac{x'\,,\,{K}/k}{{\mathfrak p}'} \big) =
\big (\frac{K /k}{{\mathfrak p}'} \big)^{v_{\mathfrak p'}(x')}$;

\smallskip
finally, we have obtained $\big (\frac{x\,,\,{K}/k}{{\mathfrak p}} \big) = 
\prod_{{\mathfrak p}' \in \Pl_0 \setminus T}\big (\frac{K /k}{{\mathfrak p}'}\big)^{-{v_{\mathfrak p'}}(x')}$.
It follows that since ${v_{\mathfrak p}}(x') = {v_{\mathfrak p}}(x)$ by (i), we can write:
$$(x') =: {\mathfrak p}^{{v_{\mathfrak p}}(x')} {\mathfrak a} = {\mathfrak p}^{{v_{\mathfrak p}}(x)} 
{\mathfrak a}, \ \ \hbox{(${\mathfrak a}$ is prime to $T$ by (ii)),} $$ 

and we have obtained (for  ${\mathfrak p} \in T$),
$\Big (\ds\frac{x\,,\,{K}/k}{{\mathfrak p}} \Big) = \Big (\frac{K /k}{\mathfrak a} \Big)^{-1}$,
where the Artin symbol $\ds\Big (\frac{K /k}{\mathfrak a} \Big)$ is by definition built
multiplicatively from the Frobenius automorphisms of the prime divisors of ${\mathfrak a}$.
Recall that if ${\mathfrak p}\notin T$, we have
$\Big(\ds \frac{x\,,\,{K}/k}{{\mathfrak p}} \Big) = \Big(\frac{K /k}{{\mathfrak p}} \Big)^{{v_{\mathfrak p}}(x)}. $

\medskip
When we find that $x$ is a global norm in $K/k$, {\it bnfisnorm(bnfinit $(P), x)$} of PARI \cite{P}
(for $k=\Q$ and $K$ given via the polynomial $P$), gives a solution $y$;
if $x= {\rm N}(y)$ and $(x)= {\rm N}({\mathfrak A})$ for an ideal ${\mathfrak A}$ of $K$, then it is 
immediate to get numerically ${\mathfrak B}$ such that $(y) \cdot {\mathfrak B}^{1-\sigma} = {\mathfrak A}$.
This was used for the Example \ref{82}.}
\end{remark}

One can find numerical computations, densities results, notions of ``governing fields''
and heuristic principles in many papers like \cite{Gr3$'$}, \cite{Mo1}, \cite{Mo2}, \cite{St2},
\cite{Wi}, \cite{Y2}, \cite{Ge4}, etc. We think that the local framwork
given by the algorithm may confirm these heuristic results since normic symbols are independant
(up to the product formula) and take uniformly all values with standard probabilities.

\subsection{$p$-triviality criterion for $p$-class groups in a $p$-extension}

When $K/k$ is cyclic of $p$-power degree, the triviality of $\Cl_K^{+}$, equivalent to
$\Cl_K^{+\, G}=1$, is easily characterized from the Chevalley's formula \eqref{eq28} and gives:
$$\frac{\# \Cl_k^{+}\cdot \prod_{{\mathfrak p} \in \Pl_{k,0}} e_{\mathfrak p}}
{[K : k]\, (E_k^+  : E_k^+ \cap {\rm N}(K^\times))} = \# {\rm N}(\Cl_K^{+}) \!\cdot\!
\frac{\prod_{{\mathfrak p} \in \Pl_{k,0}} e_{\mathfrak p}}
{[K : K \cap H_k^{+}] \, (E_k^+  : E_k^+ \cap {\rm N}(K^\times))} = 1, $$
which leads to the two conditions $H_k ^{+} 
\subseteq K \  \& \  (E_k^+  : E_k^+ \cap {\rm N}(K^\times))=\# \Omega_{K/k}$,
which is coherent with the fact that the genera field $H_{K/k}^+$ is $K$
(see \eqref{eq29} and \eqref{eq30}, \S\,\ref{genera}). 
Any generalization ($S$-class groups with modulus, quotients by a sub-module ${\mathcal H}$) 
is left to the reader.

\medskip
The following result gives, when the $p$-group $G$ is not cyclic, a charcterisation of the condition 
$\Cl_K^S = 1$ despite the fact that the usual Chevalley's formula does not exist in the non-cyclic case; 
so this involves more deep invariants as the knot group $\hbox{\Large$\kappa$}$ and
the $p$-central class field $C_{K/k}^S$ 
(i.e., the largest subextension of $H_K^S/K$, Galois over $k$, such that
${\rm Gal}(C_{K/k}^S / K)$ is contained in the center of ${\rm Gal}(C_{K/k}^S / k)$).

\begin{theorem}\label{central}
Let $K/k$ be a $p$-extension with Galois group $G$ (not necessarily Abelian), let $S$ be 
a finite set of non-complex places of $k$ and let $\Cl_K^S$ be the $p$-Sylow subgroup 
of the $S$-class group of $K$.
Then $\Cl_K^S = 1$ if and only if the following three conditions are satisfied, where 
$J_K$ is the id\`ele group of $K$:

\smallskip
(i) $H_k^S \subseteq K$,

(ii) $(E_k^S: E_k^S \cap {\rm N}_{K/k}(J_K)) = 
\ds \frac {\prod_{{\mathfrak p} \notin S}\  e_{\mathfrak p}^{\rm ab}\, \times \  
\prod_{{\mathfrak p}\in S} \ e_{\mathfrak p}^{\rm ab} \, f_{\mathfrak p}^{\rm ab} }{[K^{\rm ab}: H_k^S]}$, 
where $e_{\mathfrak p}^{\rm ab}$ (resp. $f_{\mathfrak p}^{\rm ab}$) is
the ramification index (resp. the residue degree) of the place ${\mathfrak p}$ of $k$ in the maximal
subextension $K^{\rm ab}$ of $K$, Abelian over $k$,

\smallskip
(iii) $\hbox{\Large$\# \kappa$} = (E_k^S\, \cap\, {\rm N}_{K/k}(J_K): 
E_k^S \cap {\rm N}_{K/k}(K^\times))$, where the knot group $\hbox{\Large$\kappa$}$ is by definition
$k^\times  \cap\, {\rm N}_{K/k}(J_K) \big/ {\rm N}_{K/k}(K^\times)$.
\end{theorem}

The knot group, which may be nontrivial in the non-cyclic case, measures the ``defect''
of the Hasse principle, i.e., of local norms compared to global norms.
The proof is based on the fact that $\Cl_K^S = 1$ if and only if $\Cl_K^S =  I_G\cdot \Cl_K^S$,
where $I_G$ is the augmentation ideal of $G$, because when $G$ is a $p$-group there 
exists a power of $I_G$ which is contained in $p\,\Z[G]$.
Since by duality, ${\rm H}_0(G, \Cl_K^S)$ and ${\rm H}^0(G, \Cl_K^{S\, \ast})$ have 
same order, we obtain the relation $(\Cl_K^S : I_G\cdot \Cl_K^S) = \# (\Cl_K^{S\,\ast})^G$,
which means that $[C_{K/k}^S : K] = \#  (\Cl_K^{S\,\ast})^G$; 
thus we recover the condition by using the classical fixed point theorem for finite $p$-groups.
From the formula giving $[C_{K/k}^S : K]$ (cf. \cite[Theorem IV.4.7]{Gr2}), 
we deduce the three conditions of the theorem.

\smallskip
For a detailed proof, see \cite[\S\,IV.4.7.4]{Gr2} giving a historic of the genera and central classes theories
from works of Scholz, Fr\"ohlich, Furuta, Gold, Garbanati, Jehne, Miyake, Razar, Shirai, and many others; 
see \cite{L4} for an history of genus theory and related results.

\begin{remark} Condition (iii) is empty when $G$ is cyclic (Hasse principle), or when $\hbox{\Large$\kappa$} = 1$.
The condition $\hbox{\Large$\kappa$} = 1$ can be checked in the Abelian case via Razar's criterion, see \cite{Ra}; 
on the contrary it becomes nontrivial in the other cases so that, in practice, there does not exist any easy numerical 
criterion for the triviality of the $p$-class group in a non-cyclic $p$-extension.

\smallskip
In the particular case $k = \Q$, $S = \emptyset$, condition (i) is empty, condition (ii), equivalent to 
$\prd_{{\mathfrak p} \in \Pl_0}e_{\mathfrak p}^{\rm ab} = [K^{\rm ab} : \Q]$,
is easy to check, and condition (iii) is equivalent to $\hbox{\Large$\kappa$} = 1$;
this implies that for $k = \Q$ with the narrow sense, the above problem
is essentially reduced to that of the Hasse principle.
\end{remark}

\section{Relative $p$-class group of an Abelian field of prime to $p$ degree}\label{sect5}
We fix a prime number $p$. To simplfy, we suppose $p>2$.

\smallskip
We shall apply the above results of Sections $3$ and $4$ to study the Galois structure of the relative 
$p$-class group of an imaginary Abelian extension $k / \Q$, {\it of prime to $p$ degree}, using both 
the genera theory with characters in a suitable extension $K/k$, cyclic of degree $p$,
and the ``principal theorem'' of Thaine--Ribet--Mazur--Wiles--Kolyvagin in $k$ \cite{MW}.

\smallskip
This section, based on \cite[(1993)]{Gr11}, emphasizes an interesting 
phenomena which is, roughly speaking, that when one grows up in suitable $p$-extensions $K/k$, 
the $p$-class group of $K$ becomes ``more regular'' and gives informations on the $p$-class group 
of the base field $k$; the most spectacular case being Iwasawa theory in $\Z_p$-extensions \cite{Iw}
giving for instance (under the nullity of the $\mu$-invariant) Kida's formula for the $\lambda^-$-invariants 
in finite $p$-extensions $K/k$ of CM-fields, which is nothing else than a ``genera theory'' comparison 
of $p$-ranks of relative class groups ``at infinity'', i.e., in $K_\infty/k_\infty$ where $k_\infty$ and 
$K_\infty$ are the cyclotomic $\Z_p$-extensions of $k$ and $K$, respectively (see various approaches 
in  \cite{Iw}, \cite{Ki}, \cite{Sin}). For instance, when $K/k$ is cyclic of degree $p$ 
one gets for the whole $\lambda$-invariants, assuming $K \cap k_\infty = k$ (\cite[Theorem 6 (1981)]{Iw}):
$$\lambda(K) -1 = p  \cdot  (\lambda(k)-1) + (p-1)  \cdot  
 (\chi(G, E_{K_\infty}) +1) + \sm_{w}\big (e_w (K_\infty/k_\infty) - 1\big), $$
 
where $w$ ranges over all non-$p$-places of $K_\infty$, where $p^{\chi(G, E_{K_\infty})}$ is the 
Herbrand quo\-tient $\frac{{\rm H}^2(G, E_{K_\infty})}{{\rm H}^1(G, E_{K_\infty})}$ of the group
$E_{K_\infty}$ of units of $K_\infty$ (similar situation as for Chevalley's
formula which needs the knowledge of the Herbrand quotient of $E_K$)
and where $e_w (K_\infty/k_\infty)$ is the ramification index of $w$ in $K_\infty/k_\infty$.

\smallskip
This aspect, in $p$-extensions different from $\Z_p$-extensions, is probably not sufficiently thorough.

\subsection{Abelian extensions of $\Q$ and characters}
Now we fix a prime number $p>2$. 
Let $\Q^{\rm ab}$, seen in $\C_p$ (the completion of an algebraic closure of $\Q_p$), 
be the maximal Abelian extension of $\Q$ (as we know, it is the compositum of all 
cyclotomic extensions of $\Q$), and let $G^{\rm ab} := {\rm Gal}(\Q^{\rm ab}/\Q)$. 

\smallskip
Let $\Psi$ be the group of $\C_p$-irreducible characters $\psi : G^{\rm ab} \too \C_p^\times$
of finite order, and let ${\mathcal X}$ be the set of $\Q_p$-irreducible characters
$\chi$ (such a character $\chi$ is the sum of the $\Q_p$-conjugates $\psi_i$ of a 
character $\psi \in \Psi$; then we say that these conjugates $\psi_i$ divide $\chi$, 
denoted $\psi_i \mid \chi$).

\smallskip
We denote by $k_\chi$ (cyclic over $\Q$) the subfield of $\Q^{\rm ab}$ fixed  by the kernel 
${\rm Ker}(\chi)$ of $\psi$ and by $R_\chi$ the ring of values of $\psi$ over $\Z_p$ 
($k_ \chi$, ${\rm Ker}(\chi)$, $R_\chi$ do not depend on the choice of the conjugate of $\psi$, 
whence the notation); furthermore, these objects only depend on the {\it $\Q$-irreducible character} 
$\rho$ above $\psi$ or $\chi$ ($\rho$ is the sum of all $\Q$-conjugates of $\psi$ then a sum 
of some $\chi$). The degree of $k_\chi/\Q$ is equal to the order of $\psi \mid \chi$.

\smallskip
The ring $R_\chi$ is a cyclotomic local ring 
whose maximal ideal is denoted ${\mathfrak M}_\chi$; more precisely, if $\psi \mid \chi$ is of order
$d\,p^n$, $p \nmid d$, $n \geq 0$, then $R_\chi = \Z_p[\xi_{d\,p^n}] = \Z_p[\xi_d] [\xi_{p^n}]$, 
where $\xi_d$ and $\xi_{p^n}$ are primitive $d$th and $p^n$th roots of unity, respectively; the prime $p$
is unramified in $\Q_p(\xi_d)/\Q_p$ and totally ramified in $\Q_p(\xi_{p^n})/\Q_p$ of degree $(p-1)\,p^{n-1}$, 
so that we get
$${\mathfrak M}_\chi^{(p-1) p^{n-1}} \! = p\cdot R_\chi . $$

Let 

\medskip
\centerline {${\mathcal X}_0 := \{\chi \in {\mathcal X}, \ \ \hbox{$\psi \mid \chi$ is of order prime to $p$} \}$}

and 

\centerline{${\mathcal X}_p := \{\chi \in {\mathcal X}, \ \ \hbox{$\psi \mid \chi$ is of $p$-power order} \}.$}

\medskip
We verify that ${\mathcal X} = {\mathcal X}_0 \cdot {\mathcal X}_p$ since
for any $\chi \in {\mathcal X}$ and $\psi \mid \chi$, we have the unique factorization 
$\psi = \psi_0 \cdot \psi_p$ where $\psi_0$
is of order prime to $p$ and $\psi_p$ is of $p$-power order, then $\chi = \chi_0^{} \cdot \chi_p$,
where $\psi_0 \mid \chi_0^{}$ and $\psi_p \mid \chi_p$, since $\Q_p(\xi_d)/\Q_p$
and $\Q_p(\xi_{p^n})/\Q_p$ are linearly disjoint over $\Q_p$. Note that $\chi_p$ is also
the $\Q$-irreducible character deduced from $\psi_p$ since $\Q$-conjugates and $\Q_p$-conjugates 
of $\psi_p$ coincide.
The local degree $f_{\chi_0^{}} := [\Q_p(\xi_d): \Q_p]$ is the residue degree of $p$ in $\Q(\xi_d)/ \Q$. 

\smallskip
We say that $\chi$ is even (resp. odd) if $\psi(s_{-1}) =1$ (resp. $\psi(s_{-1}) =-1$), where 
$s_{-1}$ is the complex conjugation. We denote  by
${\mathcal X}^\pm$, ${\mathcal X}_0^\pm$ and ${\mathcal X}_p^\pm$ the corresponding
sets of even or odd characters (note that since $p\ne 2$, ${\mathcal X}_p = {\mathcal X}_p^+$). 

\smallskip
For any subfield $K$ of $\Q^{\rm ab}$ we denote by ${\mathcal X}_K$ 
(then ${\mathcal X}_{K,0}$, ${\mathcal X}_{K,p}$) the set of characters of $K$
(i.e., such that ${\rm Gal}(\Q^{\rm ab}/K) \subseteq {\rm Ker}(\chi)$ or $k_\chi \subseteq K$).

\subsection{The universal $\chi$-class groups ($\chi \in {\mathcal X}$, $p>2$)}

Let $\Cl_F$ denotes the $p$-class group of any field $F \subset \Q^{\rm ab}$ (since $p>2$,
we have implicitely the ordinary sense). Let $\chi \in {\mathcal X}$.

\smallskip
(i) If $\chi = \chi_0^{} \in {\mathcal X}_0$, let  $e_{\chi_0} = 
\frac{1}{[k_{\chi_0} : \Q]} \sm_{s \in {\rm Gal}(k_{\chi_0} /\Q)} {\chi_0}(s^{-1})\,s$
be the idempotent of $\Z_p [{\rm Gal}(k_{\chi_0} /\Q)]$ associated with $\chi = \chi_0^{}$; so
we have $\Z_p [{\rm Gal}(k_{\chi_0} /\Q)] \cdot e_{\chi_0} \simeq R_{\chi_0^{}}$.

\smallskip
Then we define the $\chi_0^{}$-class group as the corresponding semi-simple 
component of $\Cl_{k_{\chi_0}}$ defined by
$$\Cl_{\chi_0} := \Cl_{k_{\chi_0}}^{e_{\chi_0}}. $$ 

(ii) If $\chi = \chi_0^{} \cdot \chi_p$ with $\chi_0^{} \in {\mathcal X}_0$ and $\chi_p \in {\mathcal X}_p$,
$\chi_p \ne 1$, let $k'$ be the unique subfield of $k_\chi$ 
such that $[k_\chi : k']=p$ (we have $k' = k_{\chi_0^{}}$ only if $\chi_p$ is of order $p$); thus the 
arithmetical norm ${\rm N}_{k_\chi / k'}$ induces the following exact sequence of 
$R_{\chi_0^{}}$-modules defining $\Cl_\chi$:
$$1 \too \Cl_\chi \tooo \Cl_{k_\chi}^{e_{\chi_0}}  \mathop{\tooo}^{{\rm N}_{k_\chi \! /\! k'}}
\Cl_{k'}^{e_{\chi_0}} \too 1, $$

the surjectivity being obvious because $k_\chi$ is the direct compositum over $\Q$ of $k_{\chi_0^{}}$
and $k_{\chi_p}$ which is a cyclic $p$-extension of $\Q$, thus totally ramified
at least for a prime number, whence $k_\chi / k'$ ramified. Since $\Cl_\chi$ is anihilated by
$e_{\chi_0}^{}$ and by ${\rm N}_{k_\chi  / k'}$ which corresponds to $1+ \sigma+ \cdots + \sigma^{p-1}$
in the group algebra of ${\rm Gal}(k_{\chi}/ k') =: \langle \sigma \rangle$,  $\Cl_\chi$ is
canonically a $R_\chi$-module (and not only a $R_{\chi_0^{}}$-module).

\smallskip
This defines, by an obvious induction in $ k_{\chi} / k_{\chi_0^{}}$, the universal family of 
components $\Cl_\chi$ for all $\chi \in {\mathcal X}$ for which we have the following 
formulas for any cyclic extension $K/\Q$ of degree $d\cdot p^n$, $p \nmid d$, $n\geq 0$:
\begin{equation}\label{eq37}
\begin{aligned}
\# \Cl_K  &= \prd_{\chi_0^{} \in  {\mathcal X}_{K, 0}} \# \Cl_K^{e_{\chi_0^{}}} , \\
\# \Cl_K^{e_{\chi_0^{}}} &= \prd_{i=0}^n \# \Cl_{\chi_i} ,\ \ \forall \chi_0^{}  \in {\mathcal X}_{K, 0},
\end{aligned}
\end{equation}

where, for each $\chi_0^{} \in {\mathcal X}_{K, 0}$, $\chi_i = \chi_0^{} \cdot \chi_{p, i}$, 
where $\chi_{p, i}$ is above $\psi_p^{p^{n-i}}$ for $\psi_p \mid \chi_p^{}$.

\smallskip
We denote by $\omega$, of order $p-1$, the Teichm\"uller character for $p>2$; we have 
$k_\omega = \Q(\zeta_p)$ where $\zeta_p$ is a primitive $p$th root of unity and by definition 
$\omega(\zeta_p\to \zeta_p^a) \equiv a \pmod p$ for $a=1,\ldots,p-1$. 

\medskip
With these definitions, we can give the statement of the ``principal theorem'' of 
Thaine--Ribet--Mazur--Wiles--Kolyvagin \cite{MW} in the particular context of  imaginary fields $K$
for  the relative class groups $\Cl_K^-$, hence with odd characters.

\begin{theorem} \label{MWK}
Let $p\ne 2$ and let $\chi = \chi_0^{} \cdot \chi_p \in {\mathcal X}^-$. We assume that $\chi_0^{} \ne \omega$
when $k_\chi$ is the cyclotomic field $\Q(\zeta_{p^n})$ (otherwise $\Cl_{\chi}=1$).
For $\psi \mid \chi$, let $b_\chi$ be the ideal $B_1(\psi^{-1})\cdot R_\chi$ 
where $B_1(\psi^{-1})$ is the generalized Bernoulli number of the character~$\psi$. Then we have
$\# \Cl_{\chi} = \# (R_\chi / b_\chi)$.
\end{theorem}

But as it is well known, this result does not give the structure of $\Cl_{\chi}$ as $R_\chi$-module;
indeed, if $b_\chi = {\mathfrak M}_\chi^t$, we may have the general structure:
$$\hbox{$\Cl_{\chi} \simeq \plus_{i=1}^e R_\chi / {\mathfrak M}_\chi^{t_i}$, $\ \  1 \leq t_1 \leq \cdots \leq t_e$, 
$\ \ e \geq 0$, $\ \ \sm_{i=1}^e t_i = t$.} $$

For instance, $\Cl_\chi$ is $R_\chi$-monogenic if and only if $e=1$.

\subsection{Definition of admissible sets of prime numbers}
Still for $p\ne 2$ and  $\chi_0^{} \in {\mathcal X}_0^-$, $\chi_0^{} \ne \omega$, consider the cyclic field
$k := k_{\chi_0^{}}$ for which $\Cl_{\chi_0^{}} = \Cl_k^{e_{\chi_0^{}}}$, where 
$e_{\chi_0^{}} = \frac{1}{[k_{\chi_0}  : \Q]} \sm_{s \in {\rm Gal}(k_{\chi_0} /\Q)} \chi_0^{} (s^{-1})\,s$. 
We intend to apply the previous sections of this paper on genera theory to obtain informations 
on the structure of $\Cl_{\chi_0^{}}$.

\begin{definitions} (i) For any $t \geq 1$, let ${\mathcal S}_t$ be the familly of sets
$\{\ell_1, \ldots, \ell_t\}$ of $t$ prime numbers fulfilling the following conditions 
(for given $\chi_0^{} \in {\mathcal X}_0^-$ and $\psi_0 \mid \chi_0^{}$): 

\medskip
\hspace{0.5cm} $\ell_i \equiv 1 \pmod p$, for $i=1, \ldots , t$ (i.e., $p \mid [\Q(\zeta_{\ell_i}) : \Q]$);

\medskip
\hspace{0.5cm} $\psi_0(\ell_i)=1$, for $i=1, \ldots , t$ (i.e., $\ell_i$ totally splits in $k = k_{\chi_0^{}}$). 

\medskip
(ii) For $S \in {\mathcal S}_t$, let $\Phi_S \subset {\mathcal X}_p$ be the set
of characters  $\varphi$, of order $p$, with conductor $\ell_1 \cdots \ell_t$ (that is to say,
$k_ \varphi \subseteq \Q(\zeta_{\ell_1\cdot \cdots \cdot \ell_t}^{})$ is of conductor 
$\ell_1 \cdots \ell_t$, whence if $k_i$ is the unique subfield of $\Q(\zeta_{\ell_i})$ of degree $p$, then
$k_ \varphi$ is a subfield of degree $p$ of the compositum $k_1 \cdots k_t$ and $k_ \varphi$ is not in 
a compositum of less than $t$ fields $k_i$).

\medskip
(iii) The character $\varphi \in \Phi_S$ is said to be $\chi_0^{}$-{\it admissible} if 
$b_{\chi_0^{} \cdot \varphi} = {\mathfrak M}_{\chi_0^{} \cdot \varphi}^t$ 
(see Theorem \ref{MWK} for the definition of $b_{\chi_0^{} \cdot \varphi}$). 
By extension we say that $S\in {\mathcal S}_t$ is $\chi_0^{}$-admissible if there 
exists at least a $\chi_0^{}$-admissible character $\varphi \in \Phi_S$. 

\medskip
(iv) Let $r_{\chi_0^{}}$ be the $R_{\chi_0^{}} / p\,R_{\chi_0^{}}$-dimension of $\Cl_{\chi_0^{}}/\Cl_{\chi_0^{}}^p$.
\end{definitions}

So the number $t$ is known from the computation of a Bernoulli number depending on $\varphi$ 
and it is not difficult to find $\chi_0^{}$-admissible characters $\varphi$.
Then we have proved in \cite{Gr11} the following effective result:

\begin{theorem} \label{}
Let $p\ne 2$ and let $\chi_0^{} \in {\mathcal X}_0^-$, $\chi_0^{} \ne \omega$, and let $k=k_{\chi_0^{}}$.

\smallskip
Let $S = \{\ell_1, \ldots, \ell_t\} \in {\mathcal S}_t$ be a $\chi_0^{}$-{\it admissible} set; 
then for $i=1, \ldots, t$, let ${\mathfrak l}_i$ be a prime ideal of $k$ above $\ell_i$ and let 
$h_i := \cl_k({\mathfrak l}_i)^{e_{\chi_0^{}}}$ be the image of $\cl_k({\mathfrak l}_i)$ 
in $\Cl_k^{e_{\chi_0^{}}}$.

\smallskip
Then $\Cl_{\chi_0^{}}$ is the $R_{\chi_0^{}}$-module generated by the $h_i$, $i= 1, \ldots, t$, 
and we have $r_{\chi_0^{}} \leq t$.
Taking the minimal value of $t$ yields $r_{\chi_0^{}}$.
\end{theorem}

The principle of the proof is an application of the computations of invariant classes of the 
Section \ref{sect4} in $K/k$ where $K= k_\varphi \cdot k = k_{\chi_0^{} \cdot \varphi}$ and
where $\varphi$ is the $\chi_0^{}$-admissible character of order $p$. 

\unitlength=0.75cm
$$\vbox{\hbox{\hspace{-4.7cm} \begin{picture}(11.5,2.5)
%     horizontales
\put(6.0,2.50){\line(1,0){2.8}}
\put(6.0,0.50){\line(1,0){2.8}}
%     vertical
\put(5.50,0.9){\line(0,1){1.20}}
\put(9.50,0.9){\line(0,1){1.20}}
\put(7.0,0.1){\footnotesize$d,\, p\!\nmid \! d$}
\put(9.0,2.4){$K= k_{\chi_0^{} \!\cdot \varphi}$}
\put(5.3,2.4){$k_{\varphi}$}
\put(5.4,0.40){$\Q$}
\put(9.1,0.4){$k = k_{\chi_0^{}}$}
\put(9.65,1.4){\footnotesize$G \simeq \Z/p\Z$}
\end{picture}   }} $$

\unitlength=1.0cm
We consider the $G$-module $M = \Cl_K^{e_{\chi_0^{}}}$ as a 
component of the relative class group $\Cl_K^-$; in other words, a semi-simple 
component of the $p$-class group of $K$, since from $\Cl_K^- = 
\plus_{\chi'_0 \in {\mathcal X}_k^-} \Cl_{K}^{e_{\chi'_0}}$ we have selected 
$\chi_0 \in {\mathcal X}_k^-$ and the associated filtration with characters of 
$M = \Cl_K^{e_{\chi_0^{}}}$ for which $M_1 = M^G = (\Cl_K^G)^{e_{\chi_0^{}}}$, 
$G := {\rm Gal}(K/k) \simeq \Z/p\Z$ (see \cite[(1978)]{Gr12}).

\medskip
We denote by ${\mathfrak L}_i$ the ideal of $K$ above ${\mathfrak l}_i$ (indeed, ${\mathfrak l}_i$
is totally ramified in $K/k$) and by $H_i := \cl_K ({\mathfrak L}_i)^{e_{\chi_0^{}}}$.
Then the proof consists in proving the following lemmas (see
\cite[Lemmes (1.2), (1.3), Corollaire (2.4)]{Gr11}:

\begin{lemma} The extension $j_{K/k} : \Cl_k^{e_{\chi_0^{}}} \too \Cl_K^{e_{\chi_0^{}}}$ is injective.
\end{lemma}

This comes easily from the fact that $\chi_0^{}$ is odd (the $\chi_0^{}$-components of units 
are trivial for $\chi_0^{} \ne \omega$, thus there is no capitulation of relative classes).

\begin{lemma} We have $M_1 = j_{K/k}(\Cl_k^{e_{\chi_0^{}}}) \cdot 
\langle  H_1, \ldots, H_t \rangle_{R_{\chi_0^{}}}$ and $M_1 \big / j_{K/k}(\Cl_k^{e_{\chi_0^{}}}) \simeq 
(R_{\chi_0^{}}/ p\,R_{\chi_0^{}})^{\,t}$.
\end{lemma}

This expression giving $\# M_1 = \# \Cl_k^{e_{\chi_0^{}}} \!\cdot p^{{\,t} \cdot f_{\chi_0^{}}}$,
where $f_{\chi_0^{}}$ is the residue degree of $p$ in $\Q(\xi_d)/ \Q$, is
nothing else than the $\chi_0^{}$-Chevalley's formula in $K/k$ for an odd character $\chi_0^{}$
(cf. \cite{Gr12}).

\begin{lemma} 
The character $\varphi \in \Phi_S$ is $\chi_0^{}$-admissible if and only if $M=M_1$
(in other words,  if and only if there are no exceptional $\chi_0^{}$-classes).
\end{lemma}

Thus, since $\Cl_k^{e_{\chi_0^{}}} = \Cl_{e_{\chi_0^{}}}$, we get 
$\# M := \Cl_{K}^{e_{\chi_0}} = \#\Cl_{\chi_0^{}}\! \cdot \# \Cl_{\chi_0^{} \cdot \varphi}$ 
from formula \eqref{eq37} with $n=1$.
From Theorem \ref{MWK}, we have $M=M_1$ if and only if 
$b_{\chi_0^{} \cdot \varphi} = {\mathfrak M}_{\chi_0^{} \cdot \varphi}^{\,t}$ ($\chi_0^{}$-admissibility).
From the lemmas we get ${\rm N}_{K/k}(M) = {\rm N}_{K/k}(M_1)$, hence
$\Cl_{\chi_0^{}} = \Cl_{\chi_0^{}}^p \cdot \langle  h_1, \ldots, h_t \rangle_{R_{\chi_0^{}}}$, whence
$\Cl_{\chi_0^{}} =  \langle  h_1, \ldots, h_t \rangle_{R_{\chi_0^{}}}$.

\medskip
So, for practical use, we are reduced to the known algorithm which must stop at the first step.
The ideals $b_{\chi_0^{} \cdot \varphi}$ generated by Bernoulli numbers are easily obtained
from the Stickelberger element of the field $K$:
$${\rm St}(K) := \sm_{a=1}^m \Big(\frac{K/\Q}{a}\Big)^{-1} \Big(\frac{a}{m} - \frac{1}{2}\Big)
\in {\rm Gal}(K/\Q), $$ 

where $m$ is the conductor of $K$ and $\big(\frac{K/\Q}{a}\big)$ the Artin symbol (for gcd\,$(a,m)=1$).

\smallskip
For more details see \cite{Gr11} where it is also proved that admissible sets have a nontrivial
Chebotarev density leading to the effectivness of the detemination of the structure and where 
relations with some results of Schoof \cite{Sch1} are discussed (cf. \cite[\S\S\,4, 5]{Gr11}).

\smallskip
One can then find many numerical examples in the Appendix \cite[(A)]{Gr11} by Berthier,
showing some cases of non-monogenic $\Cl_{K}^{e_{\chi_0}}$ as $R_{\chi_0^{}}$-modules.
For instance, let $k = \Q \Big(\sqrt{-541\, (37+6\,\sqrt{37})} \Big)$ (quartic cyclic over $\Q$) and $p=5$;
there exist two $5$-adic characters $\chi_0^{}$ and $\chi_0'$ for which $\Cl_{k}^{e_{\chi_0}} \simeq 
R_{\chi_0^{}}/(2-i) R_{\chi_0^{}} \plus R_{\chi_0^{}}/(2-i) R_{\chi_0^{}}$ and $\Cl_{k}^{e_{\chi'_0}}=1$
(a rare example of non-monogenic $\Cl_{k}^{e_{\chi_0}}$). See \cite{Ber} for numerical tables
where the case of even characters $\chi_0^{}$ is also illustrated.

\section{Conclusion and perspectives}
To conclude, we can say that the $p$-class group is perhaps not the only object
for the class field theory setting of a number field $k$. Indeed, we prefer the very similar finite 
$p$-group, denoted ${\mathcal T}_{k,p}$, and defined as the $p$-torsion subgroup of the 
Galois group of the maximal $p$-ramified (i.e., unramified outside $p$), non-complexified, 
Abelian pro-$p$-extension of $k$ denoted  $H_{k, p}^{\rm pra}$ in the following schema:
\unitlength=0.6cm
$$\vbox{\hbox{\hspace{-1.8cm}  \begin{picture}(11.5,5.6)
\put(8.4,4.50){\line(1,0){2.2}}
\put(4.1,4.50){\line(1,0){2.4}}
\put(4.8,2.50){\line(1,0){2.2}}

\bezier{350}(4.0,4.8)(7.6,5.4)(10.7,4.8)
\put(7.2,5.4){${\mathcal T}_{k, p}$}

\put(3.50,2.9){\line(0,1){1.20}}
\put(3.50,0.9){\line(0,1){1.20}}
\put(7.50,2.9){\line(0,1){1.20}}

\bezier{200}(3.9,0.5)(5.7,0.5)(7.4,2.2)
\put(6.4,0.8){$\Cl_{k, p}^{\rm ord}$}

\put(10.85,4.4){$H_{k, p}^{\rm pra}$}
\put(6.9,4.4){$\wt k\, H_{k, p}^{\rm ord}$}
\put(3.4,4.4){$\wt k$}
\put(7.2,2.4){$H_{k, p}^{\rm ord}$}
\put(2.85,2.4){$\wt k \!\cap\! H_{k, p}^{\rm ord}$}
\put(3.4,0.40){$k$}
\end{picture}   }} $$

\unitlength=1.0cm
where $\wt k$ is the compositum of the $\Z_p$-extensions of $k$,
$H_{k, p}^{\rm ord}$ the $p$-Hilbert class field, and $\Cl_{k, p}^{\rm ord}$ is the 
$p$-class group of $k$ (ordinary sense).

\smallskip
This finite group ${\mathcal T}_{k,p}$, connected with the Leopoldt conjecture at $p$
and the residue of the $p$-adic zeta function, has been 
studied by many authors by means of algebraic and analytic viewpoints 
(e.g., K. Iwasawa \cite{Iw}, J. Coates \cite[Appendix]{Co}, H. Koch \cite{Ko}, J-P. Serre \cite{Se2}, etc.), 
and we have done extensive 
practical studies in \cite{Gr2} from earlier publications \cite{Gr7}, \cite{Gr8}, \cite{Gr9},
and recently in a historical overview of the Bertrandias-Payan module (a quotient of ${\mathcal T}_{k,p}$) 
by means of three different approaches by J-F. Jaulent, T. Nguyen Quang Do and us
(see the details in \cite{Gr6} and its bibliography).
 
\smallskip
The functorial properties of these modules ${\mathcal T}_{k,p}$ are more canonical 
(especially in any $p$-extensions $K/k$ of Galois group $G$) with an explicit formula for 
$\#{\mathcal T}_{K,p}^G$ under the sole Leopoldt conjecture, so that a ``Chevalley's formula'' 
does exist for any $p$-extension $K/k$, see \cite[Theorem IV.3.3]{Gr2} and \cite{MoNg}; 
${\mathcal T}_{k,p}$ contains any deep information on class groups and units (using, for instance, 
reflection theorems to connect ${\mathcal T}_{k,p}$ and $\Cl_{k,p}^{\Pl_p}$ when $k$ contains
the $p$th roots of unity, \cite[Proposition III.4.2.2]{Gr2}); furthermore, it is a fundamental invariant
concerning the structure of the Galois group of the maximal $p$-ramified pro-$p$-extension of $k$, 
saying that this pro-$p$-group is free if and only if ${\mathcal T}_{k,p}=1$ 
(fundamental notion called {\it $p$-rationality} of $k$; see \cite[Theorem III.4.2.5]{Gr2}). 

\smallskip
Moreover the properties of the ${\mathcal T}_{k,p}$
in a $p$-extension are in relation with the notion of {\it $p$-primitive ramification} introduced in \cite[(1986)]{Gr9}
and largely developed in many papers on the subject (e.g., \cite{Ja2}, \cite{MoNg}). In a similar context, in connection 
with Gross's conjecture \cite{FG}, mention the {\it logarithmic class group} introduced by 
J-F. Jaulent (\cite{Ja3}, \cite{So}) governing the {\it $p$-Hilbert kernel} and the {\it $p$-regular kernel}.

\smallskip
The main property concerning these groups ${\mathcal T}_{k,p}$ is that, under the Leopoldt 
conjecture for $p$ in $K/k$ (even if $K/k$ is not Galois), the transfer map 
$j_{K/k} : {\mathcal T}_{k,p} \too {\mathcal T}_{K,p}$
(corresponding as usual to extension of ideals in a broad sense) is {\it injective} 
\cite[Theorem IV.2.1]{Gr2} contrary to the case of $p$-class groups.
Furthermore, the property of $p$-rationality we have mentionned above, has important consequences 
as is shown by Galois representations theory (e.g., \cite[(2016)]{Gre1}) or conjectural and heuristic 
aspects (e.g., \cite[(2016)]{Gr10}).

\smallskip
So we intend to make much advertise for these ${\mathcal T}_{k,p}$ since the corresponding filtration
$(M_i)_{i \geq 0}$ in a finite cyclic $p$-extension $K/k$ has not been studied to our knowledge.

\subsection*{Acknowledgments} I thank Pr. Balasubramanian Sury for his kind interest and his valuable help
for the submission of this paper. I am very grateful to the Referee for the careful reading and the suggestions for improvements of the paper.


\begin{thebibliography}{HD}


\bibitem[AT]{AT}  Artin, E., Tate, J., {\it  Class field theory}, Benjamin, New York, Amsterdam
1968; second edition: Advanced Book Classics, Addison-Wesley Publ. 
Comp., Redwood City 1990; Reprint of the 1990 second edition (2009).

\bibitem[ATZ1]{ATZ1} Azizi, A., Taous, M.,  Zekhnini, A., {\it Coclass of ${\rm Gal}(k^2_2/k)$ for some fields
$k \!=\! \Q(\sqrt{ p_1p_2 q},$ $\sqrt{-1})$ with $2$-class groups of types $(2, 2, 2)$}, Journal of Algebra 
and Its Applications 15, 2 (2016).
\url{http://scholar.google.com/citations?user=EZPtrFcAAAAJ&hl=fr}

\bibitem[ATZ2]{ATZ2} Azizi, A., Taous, M., Zekhnini, A., {\it On the strongly ambiguous classes of
some biquadratic number fields}, arXiv:1503.01992 (2015). \url{http://arxiv.org/pdf/1503.01992.pdf}

\bibitem[Bau]{Bau} Bauer, H.,  {\it Zur Berechnung der $2$-Klassenzahl der quadratischen 
Zahlk\"orper mit genau zwei verschiedenen Diskriminantenprimteilern}, 
J. Reine Angew. Math. 248 (1971), 42--46. 

\bibitem[Ber]{Ber} Berthier, Th., {\it Structure et g\'en\'erateurs du groupe des classes 
des corps quartiques cycliques sur $\Q$ (tables num\'eriques)}, Publ. Math\'ematiques de Besan\c con,
Alg\`ebre et Th\'eorie des Nombres, Ann\'ees 1992/93--1993/94, 50 pp.
\url{http://pmb.univ-fcomte.fr/1994/Berthier.pdf}

\url{http://www.sudoc.abes.fr/xslt/DB=2.1//SRCH?IKT=12&TRM=043905218}

\bibitem[Bol]{Bol} B\"olling, R., {\it On Ranks of Class Groups of Fields in Dihedral Extensions 
over $\Q$ with Special Reference to Cubic Fields}, Mathematische Nachrichten 135, 1 (1988),  275--310.
\url{https://www.researchgate.net/publication/229713427}

\bibitem[Ch1]{Ch1}  Chevalley, C., {\it  Sur la th\'eorie du corps de classes dans les corps finis
et les corps locaux} (Th\`ese), Jour. of the Faculty of Sciences Tokyo~2 (1933), 365--476.\hspace{2.0cm}

\ \ \url{http://archive.numdam.org/ARCHIVE/THESE/THESE_1934__155_/THESE_1934__155__365_0/THESE_1934__155__365_0.pdf}

\bibitem[Ch2]{Ch2}  Chevalley, C., {\it  La th\'eorie du corps de classes}, Ann. of Math. II, 41
(1940), 394--418.

\bibitem[Co]{Co}  Coates, J., {\it  $p$-adic $L$-functions and Iwasawa's theory},
In:  {\it Algebraic Number Fields}, Proc. of Durham Symposium 1975, New York-London 
(1977), 269--353.

\bibitem[FG]{FG} Federer, L.J., Gross, B.H, {\it Regulators and Iwasawa modules} 
(with an appendix by W. Sinnot,  Invent. Math. 62, 3 (1981), 443--457.

\bibitem[Fu]{Fu} Furuta, Y., {\it  The genus field and genus number in algebraic number fields},
Nagoya Math. J. 29 (1967), 281--285.
 \url{http://projecteuclid.org/download/pdf_1/euclid.nmj/1118802021}

\bibitem[Fr]{Fr} Fr\"ohlich, A., {\it  Central extensions, Galois groups and ideal class groups
of number fields}, Contemporary Mathematics 24, Amer. Math. Soc. 1983.

\bibitem[Ge1]{Ge1} Gerth III, F., {\it  On 3-class groups of certain pure cubic fields}, Bull. Austral. Math. Soc.,
72,  3  (2005), 471--476. \url{http://www.austms.org.au/Publ/Bulletin/V72P3/pdf/723-5238-GeIII.pdf}

\bibitem[Ge2]{Ge2} Gerth III, F., {\it On $p$-class groups of cyclic extensions of prime degree 
$p$ of number fields}, Acta Arithmetica,  LX.1 (1991), 85--92.
\url{http://matwbn.icm.edu.pl/ksiazki/aa/aa60/aa6013.pdf}

\bibitem[Ge3]{Ge3} Gerth III, F., {\it On $p$-class groups of cyclic extensions of prime degree 
$p$ of quadratic fields}, Mathematika 36, 1 (1989), 89--102. 
\url{http://dx.doi.org/10.1112/S0025579300013590}

\bibitem[Ge4]{Ge4} Gerth III, F., {\it The $4$-class ranks of quadratic fields},
Invent. math. 77, 3 (1984), 489--515. 

\bibitem[GK1]{GK1} Greither, C., Ku\v cera, R.,
{\it Eigenspaces of the ideal class group}, Annales Institut Fourier 64, 5 (2014), 2165--2203.
\url{https://www.researchgate.net/publication/286369796}

\bibitem[GK2]{GK2} Greither, C., Ku\v cera, R.,
{\it On a conjecture concerning minus parts in the style of Gross},
Acta Arithmetica 132, 1 (2008), 1--48.
\url{http://www.muni.cz/research/publications/764249}

\bibitem[GK3]{GK3} Greither, C., Ku\v cera, R.,
{\it Annihilators of minus class groups of imaginary Abelian fields},
Annales Institut Fourier 5, 5 (2007), 1623--1653.
\url{https://www.researchgate.net/publication/268012963}

\bibitem[GK4]{GK4} Greither, C., Ku\v cera, R.,
{\it Annihilators for the Class Group of a Cyclic Field of Prime Power Degree, II},
Canad. J. Math. 58 (2006), 580-599.

\url{http://cms.math.ca/10.4153/CJM-2006-024-2}

\bibitem[Go]{Go} Gonz\'alez-Avil\'es, C.D., {\it Capitulation, ambiguous classes and the cohomology 
of the units}, Journal f\"ur die reine und angewandte Mathematik 2007, 613 (2006), 75--97.
\url{https://www.researchgate.net/publication/2128194}
    
\bibitem[Gr1]{Gr1} Gras, G., {\it Classes g\'en\'eralis\'ees invariantes},  J. Math. Soc. Japan 46, 3 (1994), 467--476.%
\url{http://projecteuclid.org/euclid.jmsj/1227104692}

\bibitem[Gr2]{Gr2}  Gras, G., {\it Class Field Theory: from theory to practice}, SMM, Springer-Verlag 2003; 
second corrected printing 2005.\url{http://dx.doi.org/10.1007/978-3-662-11323-3} 
 
Private version 2016 (rooteng.pdf):

\url{http://www.dropbox.com/sh/64q8ezazl6b4z7d/AABhBL3Fvnf_YNTHV0GzhR8ma?dl=0}

\bibitem[Gr3]{Gr3}  Gras, G., {\it  Sur les $\ell$-classes d'id\'eaux dans les extensions cycliques relatives de degr\'e premier~$\ell$, I}, Annales de l'Institut Fourier, 23, 3 (1973), 1--48.\ \ \ 
%
\url{http://archive.numdam.org/ARCHIVE/AIF/AIF_1973__23_3/AIF_1973__23_3_1_0/AIF_1973__23_3_1_0.pdf}

\bibitem[Gr3$'$]{Gr3$'$}  Gras, G., {\it  Sur les $\ell$-classes d'id\'eaux dans les extensions cycliques relatives de degr\'e premier~$\ell$, II}, Annales de l'Institut Fourier, 23, 4 (1973), 1--44.\ \ \ 
%
\url{http://archive.numdam.org/ARCHIVE/AIF/AIF_1973__23_4/AIF_1973__23_4_1_0/AIF_1973__23_4_1_0.pdf}

\bibitem[Gr4]{Gr4}  Gras, G., {\it  Sur les $\ell$-classes d'id\'eaux des extensions non galoisiennes 
de $\Q$ de degr\'e premier impair $\ell$ \`a cl\^oture galoisienne di\'edrale de degr\'e 2\,$\ell$}, 
J. Math. Soc. Japan 26 (1974), 677--685.
\url{https://www.researchgate.net/publication/238882424}

\bibitem[Gr5]{Gr5}  Gras, G., {\it  No general Riemann-Hurwitz formula for relative $p$-class groups},
Journal of NumberTheory 171 (2017), 213--226.
\url{https://www.researchgate.net/publication/288060081}

\bibitem[Gr5$'$]{Gr5$'$}  Gras, G., {\it  Complete table concerning the paper: 
No general Riemann-Hurwitz formula for relative $p$-class groups} (2016).
\url{https://www.researchgate.net/publication/304059327}

\bibitem[Gr6]{Gr6}  Gras, G., {\it Sur le module de Bertrandias--Payan dans une $p$-extension -- 
Noyau de capitulation}, Publ. Math\'ematiques de Besan\c con,
Alg\`ebre et Th\'eorie des Nombres(2016), 25--44. 
\url{https://www.researchgate.net/publication/294194005}

\bibitem[Gr7]{Gr7}  Gras, G., {\it  Groupe de Galois de la $p$-extension ab\'elienne $p$-ramifi\'ee
maximale d'un corps de nombres}, J. reine angew. Math. 333 (1982), 86--132.
\url{https://www.researchgate.net/publication/243110955}
 
\bibitem[Gr8]{Gr8}   Gras, G., {\it  Logarithme $p$-adique et groupes de Galois}, J. reine angew.
Math. 343 (1983), 64--80. \url{https://www.researchgate.net/publication/238881752}

\bibitem[Gr9]{Gr9}   Gras, G., {\it  Remarks on $K_2$ of number fields}, J. Number Theory 23
(1986), 322--335. \url{https://www.researchgate.net/publication/243002782} 

\bibitem[Gr10]{Gr10} Gras, G., {\it Les $\theta$-r\'egulateurs locaux d'un nombre alg\'ebrique~:
Conjectures $p$-adiques}, Canad. J. Math. Vol. 68, 3 (2016), 571--624.
 \url{http://dx.doi.org/10.4153/CJM-2015-026-3}

\bibitem[Gr11]{Gr11} Gras, G., {\it Sur la structure des groupes de classes relatives. 
Avec un appendice d'exemples num\'eriques par T. Berthier}, Annales de l'Institut Fourier, 
43, 1 (1993), 1--20. \url{http://archive.numdam.org/ARCHIVE/AIF/AIF_1993__43_1/AIF_1993__43_1_1_0/AIF_1993__43_1_1_0.pdf}

\bibitem[Gr12]{Gr12} Gras, G., {\it Nombre de $\varphi$-classes invariantes. 
Application aux classes des corps ab\'eliens}, Bulletin de la Soci\'et\'e Math\'ematique 
de France, 106 (1978), 337--364. 

\ \ \url{http://archive.numdam.org/ARCHIVE/BSMF/BSMF_1978__106_/BSMF_1978__106__337_0/BSMF_1978__106__337_0.pdf}

\bibitem[Gr13]{Gr13}  Gras, G., {\it  Principalisation d'id\'eaux par extensions absolument
ab\'eliennes}, J. Number Theory 62 (1997), 403--421.

\url{http://www.sciencedirect.com/science/article/pii/S0022314X97920680}

\bibitem[Gr14]{Gr14}  Gras, G., {\it \'Etude d'invariants relatifs aux groupes des classes des corps ab\'eliens},
In: Ast\'erisque 41/42, Soci\'ete Math\'ematique de France (1977), 35--53.
\url{https://www.researchgate.net/publication/267146524}

\bibitem[Gr15]{Gr15}  Gras, G., {\it Approche $p$-adique de la conjecture de Greenberg (cas
totalement r\'eel $p$-d\'ecompos\'e)}, preprint (2016/2017).
\url{https://www.researchgate.net/publication/309731894}

\bibitem[Gre1]{Gre1} Greenberg, R.,  {\it  Galois representations with open image},
Annales de Math\'ematiques du Qu\'ebec, special volume in honor of Glenn Stevens 40, 1 (2016), 83--119.
\url{https://www.math.washington.edu/~greenber/GalRep.pdf}

\bibitem[Gre2]{Gre2} Greenberg, R., {\it On the Iwasawa invariants of totally real number fields}, 
Amer. J. Math. 98 (1976), 263--284.
\url{http://www.jstor.org/stable/2373625?seq=1#page_scan_tab_contents}

\bibitem[GW1]{GW1} Gruenberg, K.W.,  Weiss, A., {\it  Capitulation and transfer kernels},
J. Th\'eorie des Nombres de Bordeaux 12, 1 (2000), 219--226.
\url{http://archive.numdam.org/ARCHIVE/JTNB/JTNB_2000__12_1/JTNB_2000__12_1_219_0/JTNB_2000__12_1_219_0.pdf}

\bibitem[GW2]{GW2} Gruenberg, K.W.,  Weiss, A., {\it  Capitulation and transfer triples}, Proc. London Math. Soc. 3, 87 (2003), 
273--290.\url{http://journals.cambridge.org/action/displayAbstract?fromPage=online&aid=174659}
 
\bibitem[GW3]{GW3} Gruenberg, K.W.,  Weiss, A., {\it  Transfer kernels for finite groups}, J. Algebra 300, 1 (2006), 35--43.\url{http://www.sciencedirect.com/science/article/pii/S0021869305006836}

\bibitem[H]{H} Herbrand, J., {\it Le d\'eveloppement moderne de la th\'eorie des corps alg\'ebriques -
Corps de classes et lois de r\'eciprocit\'e}, M\'emorial des Sciences Math\'ematiques, Fasc. LXXV,
Gauthier--Villars, Paris 1936. \url{http://gallica.bnf.fr/ark:/12148/bpt6k39024r/f9.image}

\bibitem[HL]{HL} Harnchoowong, A., Li, W., {\it Sylow subgroups of ideal class group with moduli},
J. Number Theory 36, 3 (1990), 354--372.
\url{https://www.researchgate.net/publication/266920573}

\bibitem[I]{I} Inaba, E., {\it  \"Uber die Struktur der $\ell$-Klassengruppe zyklischer Zahlk\"orper 
von Primzahlgrad $\ell$}, J. Fac. Sci. Tokyo I, 4 (1940), 61--115.

\bibitem[Iw]{Iw} Iwasawa, K., {\it Riemann--Hurwitz formula and $p$-adic Galois representations
for number fields}, Tohoku Math. J. 33, 2 (1981),  263--288.
\url{https://www.jstage.jst.go.jp/article/tmj1949/33/2/33_2_263/_pdf}

\bibitem[Ja1]{Ja1}  Jaulent, J-F., {\it  L'arithm\'etique des $\ell$-extensions} (Th\`ese d'Etat,
Universit\'e de Franche-Comt\'e, Besan\c con), Publ. Math. Fac. Sci. Besan\c con
(Th\'eorie des Nombres), Ann\'ees 1984/86.\url{http://pmb.univ-fcomte.fr/1986/Jaulent_these.pdf}

\bibitem[Ja2]{Ja2}  Jaulent, J-F., {\it  Th\'eorie $\ell$-adique globale du  corps de classes}, J.
Th\'eorie des Nombres de Bordeaux 10, 2 (1998), 355--397.
\url{http://www.math.u-bordeaux1.fr/~jjaulent/Articles/THCDC.pdf}

\bibitem[Ja3]{Ja3}  Jaulent, J-F., {\it Classes logarithmiques des corps de nombres}, J. Th\'eor. 
Nombres Bordeaux 6 (1994), 301--325.\url{https://www.math.u-bordeaux.fr/~jjaulent/Articles/ClLog.pdf}

\bibitem[Ki]{Ki} Y. Kida, {\it $\ell$-extensions of CM-fields and cyclotomic invariants}, J. Number Theory 12 (1980), 519--528.%
\url{http://www.sciencedirect.com/science/article/pii/0022314X80900426}
 
\bibitem[Kl]{Kl} Klys, J., {\it Reflection principles for class groups} (preprint 2016).
 \url{http://arxiv.org/pdf/1605.04371.pdf}
 
\bibitem[Ko]{Ko} Koch, H., {\it  Galois Theory of $p$-Extensions}, Springer Monographs in 
Mathematics, Springer  2002. 

\bibitem[Kol]{Kol} Kolster, M., {\it The $2$-part of the narrow class group of a quadratic number field},
Ann. Sci. Math. Qu\'ebec 29, 1 (2005), 73--96.

\bibitem[KMS]{KMS} Kulkarni, M., Majumdar, D.,  Sury, B., {\it  $\ell$-Class 
 groups of cyclic extensions of prime degree $\ell$}, J. Ramanujan Math. Soc. 30, 4 (2015), 413--454.
 \url{http://www.isibang.ac.in/~sury/5class.pdf}

\bibitem[Ku]{Ku} Kurihara, M., {\it On the ideal class groups of the maximal real subfields
of number fields with all roots of unity}, J. Eur. Math. Soc. 1 (1999), 35--49.   
\url{http://link.springer.com/article/10.1007/PL00011159#page-1}

\bibitem[L]{L}  Lang, S., {\it  Algebraic Number Theory}, Addison-Wesley Publ. Comp. 1970,
corrected second printing 1986; second edition: Graduate Texts in Math. 110, Springer-Verlag 1994,
corrected third printing 2000.

\bibitem[L1]{L1} Lemmermeyer, F., {\it  Galois action on class groups}, J. Algebra 264, 2 (2003), 553--564.\ \ \ 
\ \ \ \ \url{http://www.sciencedirect.com/science/article/pii/S0021869303001224}

\bibitem[L2]{L2} Lemmermeyer, F., {\it  Class groups of dihedral extensions}, Mathematische 
Nachrichten 278, 6 (2005), 679--691.\url{http://onlinelibrary.wiley.com/doi/10.1002/mana.200310263/abstract}
\url{http://www.fen.bilkent.edu.tr/~franz/publ/mndih.pdf}

\bibitem[L3]{L3} Lemmermeyer, F., {\it The ambiguous class number formula revisited}, J. Ramanujan
Math. Soc. 28, 4 (2013), 415--421.
 \url{http://arxiv.org/pdf/1309.1071v1.pdf}
 
\bibitem[L4]{L4} Lemmermeyer, F., {\it The Development of the Principal Genus Theorem},
In: The Shaping of Arithmetic after C. F. Gauss Disquisitiones Arithmeticae, chap. VIII.3, Springer 2007,
529--561.\url{http://www.math.uiuc.edu/Algebraic-Number-Theory/0354/dpgt.pdf}

\bibitem[Mai]{Mai} Maire, Ch., {\it Une remarque sur la capitulation du groupe des classes au sens restreint},
Publ. Math. Fac. Sci. Besan\c con (Th\'eorie des Nombres), Ann\'ees 1996/97-1997/98.
\url{http://pmb.univ-fcomte.fr/1998/Maire.pdf}

\bibitem[Ma1]{Ma1} Mayer, D.C., {\it Principalization algorithm via class group structure},
J. Th\'eorie des Nombres de Bordeaux 26, 2 (2014), 415--464. 
\url{http://arxiv.org/pdf/1403.3839v1.pdf}

\bibitem[Ma2]{Ma2} Mayer, D.C., {\it The second $p$-class group of a number field}, Int. J. Number Theory 
8, 471 (2012), 471--506.\url{https://arxiv.org/abs/1403.3899}

\url{http://www.worldscientific.com/doi/abs/10.1142/S179304211250025X}

\bibitem[Miy1]{Miy1} Miyake, K. (Ed.), {\it  Class field theory -- Its
centenary and prospect} 1998, Advanced Studies in Pure Mathematics 30,
Math. Soc.  Japan 2001. \ \ \url{http://www.mathbooks.org/aspm/aspm30/aspm30-frontmatter.pdf}

\bibitem[Miy2]{Miy2}  Miyake, K., {\it  Algebraic investigations of Hilbert's theorem 94, the
principal ideal theorem and the capitulation problem}, Exp. Math. 7 (1989), 289--346.

\bibitem[Mo1]{Mo1} Morton, P., {\it Density results for the 2-classgroups of imaginary quadratic fields},
Journal f\"ur die reine und angewandte Mathematik, 332  (1982), 156--187.
\url{https://eudml.org/doc/152433}
  
\bibitem[Mo2]{Mo2} Morton, P., {\it Governing fields for the $2$-classgroup 
of $\Q(\sqrt{- q_1\, q_2\, p}) $and a related reciprocity law},
Acta Arithmetica 55 (1990), 267--290.
\url{http://matwbn.icm.edu.pl/ksiazki/aa/aa55/aa5537.pdf}

\bibitem[MoMo]{MoMo} Mouhib, A., Movahhedi, A., {\it Sur le $2$-groupe de classes 
des corps multiquadratiques r\'eels}, Jour. de Th\'eorie des Nombres de Bordeaux 17 (2005), 619--641.
\url{https://www.emis.de/journals/JTNB/2005-2/article12.pdf}

 \bibitem[MoNg]{MoNg}  Movahhedi, A., Nguyen Quang Do, T., {\it  Sur l'arithm\'etique des corps de
nombres $p$-rationnels}, S\'em. Th\'eorie des Nombres, Paris (1987/89), Progress in Math. 81, 
Birkh\"auser (1990), 155--200.
\url{https://www.researchgate.net/publication/236865321}

\bibitem[MW]{MW} Mazur, B., Wiles, A., {\it Class fields of abelian extensions of $\Q$}, Inventiones 
Mathematicae, 76, 2 (1984), 179--330. \url{https://eudml.org/doc/143124}

\bibitem[P]{P} Belabas K. and al.,  {\it Pari/gp, Version 2.5.3}, Laboratoire A2X, Universit\'e de Bordeaux I. 
\url{http://sagemath.org/}

\bibitem[Ra]{Ra}  Razar, M.J., {\it  Central and genus class fields and the Hasse norm theorem},
Compositio Math. 35 (1977), 281--298.

\bibitem[Re]{Re} R\'edei, L., {\it Ein neues zahlentheoretisches Symbol mit Anwendungen auf die Theorie
der quadratischen Zahlk\"orper}, J. Reine Angew. Math., 180 (1938), 1--43.
 
\bibitem[Sch1]{Sch1} Schoof, R.,  {\it The structure of minus class groups of abelian number fields}, 
In: C. Goldstein (ed.), S\'eminaire de Th\'eorie de Nombres, Paris 1988--1990, 
Progress in Math. 91, Birkh\"auser 1990, 185--204.
\url{http://www.mat.uniroma2.it/~schoof/dpp.pdf}

\bibitem[Sch2]{Sch2} Schoof, R.,  {\it Computing Arakelov class groups}, ``Algorithmic number theory'',
MSRI Publications 44, Cambridge University Press, Cambridge 2008, 447--495.
\url{http://www.mat.uniroma2.it/~schoof/14schoof.pdf}

\bibitem[Sch3]{Sch3} Schoof, R.,  {\it Class groups of real cyclotomic fields of prime conductor}, 
Math. Comp. 72 (2003), 913--937. 
\url{http://www.mat.uniroma2.it/~schoof/realcyc.pdf}

\bibitem[Se1]{Se1}  Serre, J-P., {\it  Corps locaux}, Actualit\'es Scientifiques
et Industrielles 1296, Hermann 1962, 1968, 1980, quatri\`eme \'edition revue et corrig\'ee 2004;
English translation: {\it  Local fields}, Graduate Texts in Math. 67, Springer-Verlag 1979,
corrected second printing 1995.

\bibitem[Se2]{Se2}  Serre, J-P., {\it  Sur le r\'esidu de la fonction z\^eta $p$-adique d'un corps
de nombres}, C.R. Acad. Sci. Paris 287, S\'erie I (1978), 183--188.

\bibitem[Sin]{Sin}  Sinnott, W., {\it On $p$-adic $L$-functions and the Riemann-Hurwitz genus formula},  
Comp. Math. 53 (1984), 3--17.
\url{http://archive.numdam.org/ARCHIVE/CM/CM_1984__53_1/CM_1984__53_1_3_0/CM_1984__53_1_3_0.pdf}

\bibitem[So]{So} Soriano, F., {\it Classes logarithmiques ambiges des corps quadratiques}, Acta Arith. 
LXXVIII.3 (1997), 201--219. \url{http://matwbn.icm.edu.pl/ksiazki/aa/aa78/aa7831.pdf}

\bibitem[St1]{St1} Stevenhagen, P., {\it R\'edei-matrices and applications}, Number theory (Paris, 1992--1993), 245--259, London Math. Soc. Lecture Note Ser. 215, Cambridge Univ. Press, Cambridge 1995. 

\url{http://dx.doi.org/10.1017/CBO9780511661990.015}

\bibitem[St2]{St2} Stevenhagen, P.,  {\it Ray class groups and governing fields},
Publ. Math. Fac. Sci. Besan\c con (Th\'eorie des Nombres),  Fasc 1, Ann\'ees 1988/1989.
\url{http://pmb.univ-fcomte.fr/1989/Stevenhagen.pdf}

\bibitem[Su]{Su} Suzuki, H., {\it  A generalization of Hilbert's theorem 94}, Nagoya Math. J. 121 (1991), 161--169.

\bibitem[SW]{SW} Schoof, R.,  Washington, L.C., {\it Visibility of ideal classes},
J. Number Theory 130, 12, 2715--2731. 
\url{http://www.sciencedirect.com/science/article/pii/S0022314X1000185X}

\bibitem[Ter]{Ter} Terada, F.,  {\it  A principal ideal theorem in the genus fields}, 
Tohoku Math. J. 23, 2 (1971), 697--718.

\bibitem[Wa]{Wa} Washington, L.C., {\it Introduction to cyclotomic fields}, Graduate Texts in Math. 83, 
Springer enlarged second edition 1997.

\bibitem[Wi]{Wi} Wittmann, Ch.,  {\it $p$-class groups of certain extensions of degree $p$},
Math. of Computation 74, 250 (2004), 937--947.

\url{http://www.ams.org/journals/mcom/2005-74-250/S0025-5718-04-01725-9/}
 
\bibitem[Y1]{Y1} Yue, Q., {\it The generalized R\'edei-matrix}, Mathematische Zeitschrift 261, 1 (2008), 23--37.

\bibitem[Y2]{Y2} Yue, Q., {\it  $8$-ranks of class groups of quadratic number fields and their densities},
Acta Mathematica Sinica 27, 7 (2011), 1419--1434.
\url{https://www.researchgate.net/publication/226464839}

\end{thebibliography}
\end{document}